\documentclass[a4paper,10pt]{article}
\usepackage{amsmath}
\usepackage{amsfonts}
\usepackage{amsthm}
\usepackage{latexsym}
\usepackage{amssymb}
\usepackage[all]{xy}
\usepackage{psfrag}
\usepackage{epsfig}
\usepackage{color}
\usepackage[all]{xy}
\usepackage{hyperref}
\usepackage{makeidx}
\makeindex

\hypersetup{colorlinks, citecolor=black, filecolor=black, linkcolor=black}

\newcommand{\C}{\mathbb{C}}

\newcommand{\Z}{\mathbb{Z}}

\newcommand{\CA}{\mathcal{A}}
\newcommand{\CB}{\mathcal{B}}

\newcommand{\CD}{\mathcal{D}}
\newcommand{\CE}{\mathcal{E}}
\newcommand{\CF}{\mathcal{F}}
\newcommand{\CG}{\mathcal{G}}
\newcommand{\CH}{\mathcal{H}}
\newcommand{\CI}{\mathcal{I}}

\newcommand{\CK}{\mathcal{K}}

\newcommand{\CO}{\mathcal{O}}
\newcommand{\CP}{\mathcal{P}}

\newcommand{\CR}{\mathcal{R}}
\newcommand{\CS}{\mathcal{S}}
\newcommand{\CT}{\mathcal{T}}

\newcommand{\CV}{\mathcal{V}}

\newcommand{\FP}{{\mathfrak{p}}}
\newcommand{\FQ}{{\mathfrak{q}}}

\newcommand{\qcoh}{\mathcal{QC}\mathfrak{oh}}
\newcommand{\coh}{\mathcal{C}\mathfrak{oh}}
\newcommand{\Mod}{\mathcal{M}\mathfrak{od}}

\newcommand{\Rep}{\mathcal{R}\mathfrak{ep}}

\DeclareMathOperator{\chr}{char}
\DeclareMathOperator{\End}{End}
\DeclareMathOperator{\Id}{Id}
\DeclareMathOperator{\Spc}{Spc}
\DeclareMathOperator{\Spec}{Spec}
\DeclareMathOperator{\supp}{supp}
\DeclareMathOperator{\supph}{supph}

\newcommand{\beast}{\begin{eqnarray*}}
\newcommand{\eeast}{\end{eqnarray*}}

\newcommand{\lan}{\langle}
\newcommand{\ran}{\rangle}

    \newtheorem{theorem}{Theorem}[section]
    \newtheorem{lemma}[theorem]{Lemma}
    \newtheorem{proposition}[theorem]{Proposition}
    \newtheorem{corollary}[theorem]{Corollary}
	\theoremstyle{definition}
	\newtheorem{definition}[theorem]{Definition}
	
	\theoremstyle{remark}
	\newtheorem{remark}[theorem]{Remark}
	\newtheorem{example}[theorem]{Example}

\title{Triangular spectrum of some triangulated categories}
\author{Umesh V. Dubey\thanks{Institute of Mathematical Sciences, CIT Campus,
Tharamani, Chennai 600113, India} \and Vivek M. Mallick\thanks{Centre de Recerca
Matem\`atica, Campus de Bellaterra, Edifici C, 08193 Bellaterra (Barcelona),
Spain}}
\date{}
\begin{document}
\maketitle
\section{Introduction}

This paper studies the prime spectrum of two tensor triangulated
categories.  Triangulated categories have been one of the most
influential objects in mathematics. Introduced by Grothendeick and
Verdier to study Serre duality in a relative setting, this idea was
soon developed by Verdier and Illusie who studied the derived category
of the abelian category of coherent sheaves, and the triangulated
category of perfect complexes respectively. Slowly the abstract
homological construction of triangulated categories permeated into
other subjects like topology, modular representation theory and even
Kasparov's KK theory. Balmer's paper \cite{PB2} gives a nice summary
of the elegant history.
  
In algebraic geometry, triangulated categories mostly appear as the
derived category of the abelian category of coherent sheaves on a
variety and as the category of perfect complexes on a variety. The
later category, as was observed by Neeman \cite{Neeman1}, are just the
compact objects of the derived category of the abelian category of
quasi-coherent sheaves (in case the scheme is quasi compact and
separated). From now on we shall call the derived category of the
category of coherent sheaves to be the derived category of the
variety. Gabriel \cite{Gabriel} and Rosenberg \cite{Rosenberg} proved
that the category of quasi coherent sheaves completely determine the
underlying variety. Bondal and Orlov \cite{Orlov2} proved that a
smooth variety can be reconstructed from the derived category of
coherent sheaves provided that either the canonical bundle or the
anti-canonical bundle is ample. But the ampleness condition here is
crucial, as Mukai \cite{Mukai} gave an example of two nonisomorphic
varieties whose derived categories are equivalent.
  
Balmer \cite{PB2} proved that in addition to the triangulated
structure on a derived category, if we also consider the tensor
structure induced by the tensor structure in the category of coherent
sheaves, we have enough information to reconstruct the variety. He
gave a method to reconstruct, by constructing ``the Spec'' of the
tensor triangulated category. The definition of Spec is quite general
and applies to any tensor triangulated category. One question that
naturally arises is how good is Spec as an invariant of the tensor
triangulated category? It turns out that there do exist pairs of
tensor triangulated categories which have isomorphic Specs (isomorphic
as ringed spaces). We give two such examples. This raises the question
of whether one can define a finer geometric invariant. Some possible
answers are discussed in the first author's thesis with HBNI.
   
In section \ref{basics}, we recall the definition of Spec. We also
recall some facts about $G$ sheaves and prove some lemmas which shall
be useful in the next section.
   
In section \ref{sec:dercatshv} we compute the Spec of the derived
category of the abelian category of coherent $G$-equivariant sheaves
on some smooth quasi-projective scheme $X$. Since the scheme is quasi
projective there exists an orbit space, see \cite{Mumford}, which we
denote as $X/G$.  As $G$ is a finite group and hence we get a finite
map $\pi : X \rightarrow X / G$ which is also a perfect morphism.
Recall that a $G$ equivariant or $G$ linearized sheaf is defined as
follows
\begin{definition} \label{def:gsheaf}
A $G$-sheaf (or $G$-equivariant sheaf or an equivariant sheaf with
respect to the group $G$) on $X$ is a sheaf $\CF$ together with isomorphisms
$\rho_g : \CF \to g^*\CF$ for all $g \in G$ such that following diagram
\[
\xymatrix{ \CF \ar[r]^{\rho_h} \ar[drr]_{\rho_{gh}} &
  h^*\CF \ar[r]^{h^*\rho_g} & \ar@{=}[d] h^*g^*\CF \\
  & & (gh)^*\CF }
\]
is commutative for any pair $g, h \in G$. A $G$-sheaf is a pair $(\CF,
\rho )$.
\end{definition}
The category of coherent $G$-sheaves is denoted as $\coh^G(X)$ and for
simplicity we denote by $\CD^G(X)$, the bounded derived category of
coherent $G$-sheaves. Consider the affine map $\pi : X \rightarrow X /
G$. Then $\CD^G(X)$ admits a functor from $\CD^{per} (X/G)$,
\[
\pi^* : \CD^{per} (X/G) \rightarrow \CD^G(X) \text{.}
\]
Since we consider only quasi projective varieties therefore the
perfect complexes are nothing but bounded complexes of vector bundles.

We prove the following theorem.

\begin{theorem}
  Assume that the scheme $X$ is a smooth quasi projective variety over
  a field $k$ of characteristic $p$ with an action of a finite group
  $G$. Assume that the order of $G$ is coprime to $p$. The induced map
  \[
  \Spec(\pi^*) : \Spec(\CD^G(X)) \rightarrow \Spec(\CD^{per}(X/G))
  \]
  is an isomorphism of locally ringed spaces.
\end{theorem}
The proof involves some computation using results from representation
theory.

Superschemes, studied by Manin and Deligne (see for example
\cite{Manin1}), are also an important object of study in modern
algebraic geometry, specially due to applications in physics. The
following definition of split superscheme is given in [ Pg. 84-85,
Manin\cite{Manin3}].
\begin{definition}
  \begin{enumerate}
  \item A ringed space $( X , \CO_X)$ is called \emph{superspace} if
    the ring $\CO_X(U)$ associated to any open subset $U$ is
    supercommutative and each stalk is local ring. A \emph{superspace}
    is called \emph{superscheme} if in addition the ringed space $(X ,
    \CO_{X,0})$ is a scheme and $\CO_{X,1}$ is a coherent sheaf over
    $\CO_{X,0}$ (where the subscript 0 denotes the even part and the
    subscript 1 denotes the odd part). We shall denote by $J_X$ the
    ideal sheaf generated by $\CO_{X,1}$ inside $\CO_X$.
  \item A superscheme $(X , \CO_X)$ is called \emph{split} if the
    graded sheaf $Gr \CO_X$ with mod 2 grading is isomorphic as a
    locally superringed sheaf to $\CO_X$. Here the graded sheaf
    \[
    Gr \CO_X := \oplus_{i \geq 0} J_X^i/J_X^{i + 1} \mbox{ where }
    J_X^0 := \CO_X .
    \]
  \end{enumerate}
\end{definition}
Manin has also given example of superschemes which are not split
superschemes.  An important example of a split superscheme is super
projective space $\mathbb{P}^{n|m}$.  We consider the triangulated
category $\CD^{per}(X)$ of `` perfect complexes '' (the definition
being modified appropriately in the super setting) on this
superscheme.

\begin{theorem}
  Let $X$ be a split superscheme. Let $X_0 = (X, \CO_{X, 0})$ be the
  0-th part of this superscheme. Here $X_0$ is by definition a
  scheme. Then we have an isomorphism of locally ringed spaces
  \[
  f : X_0 \rightarrow \Spec(\CD^{per}(X)) \text{.}
  \]
\end{theorem}
The proof of homeomorphism adapts the classification of thick tensor
ideals due to Thomason\cite{Thomason2} as demonstrated by
Balmer\cite{PB2}. Again, following Balmer\cite{PB2} we use the
generalized localization theorem of [Theorem 2.1,
Neeman\cite{Neeman1}] to finish the proof.

\noindent \textbf{Acknowledgement : } We would like to take this
opportunity to thank Prof. Kapil Paranjape and Prof. V. Srinivas for
their encouragement and insightful comments.

% modify intro later
\section{Preliminaries} \label{basics} In this section we shall recall
various basic definitions and facts which are used explicitly or
implicitly later.
\subsection{Categorical definitions}
As we are borrowing many definitions and results from Balmer's
papers\cite{PB1}\cite{PB2} so we shall work only with an essentially
small categories i.e. categories equivalent to a small category. We
recall first some basic definitions,
\begin{definition}[Semi simple abelian category]
  An abelian category is called \emph{semisimple} if every short exact
  sequence splits.
\end{definition}
\begin{definition}[Triangulated category]
  An additive category $\CD$ with a functorial isomorphism $T$,
  (called \emph{translation} or \emph{shift},) and a collection of
  sextuple $(a, b, c, f, g, h )$ with objects $a$, $b$, $c$ and
  morphisms $f$, $g$, $h$, called \emph{distinguished triangles},
  satisfying certain axioms, (cf.  \cite{Verdier}\cite{Hartshorne2},)
  is called \emph{triangulated category}. Traditionally the image of
  any object, say $a$, via functor $T^i$ is denoted as $a[i]$ and a
  distinguished triangle is denoted in a similar way as short exact
  sequences: $a \to b \to c \to a[1]$.

  An additive functor $F : \CD_1 \to \CD_2$ between two triangulated
  categories $\CD_1$ and $\CD_2$ is called a \emph{triangulated
    functor} if it commutes with the translation functor i.e. $ F
  \circ T = T \circ F$ and takes distinguished triangles to
  distinguished triangles, i.e.  $F(a) \to F(b) \to F(c) \to F(a)[1]$
  is distinguished for every distinguished triangles $a \to b \to c
  \to a[1]$.
\end{definition}
\begin{example}
  Let $\CA$ be an abelian category and $K^*(\CA)$ (resp.
  $\CD^*(\CA)$), for ($* = -, + \text{ or } b$), be the homotopy
  (resp. derived) category of an abelian category $\CA$. Then both
  additive categories are triangulated categories, see \cite{Manin2}
  for proof. In particular we are interested in the cases when $\CA =
  \coh^G(X)$ for some variety $X$ with an action of some finite group
  $G$; see subsection \ref{G-shv} for more details. When group $G$ is
  trivial then $\CA$ is an abelian category of coherent sheaves on
  variety $X$.  Another class of examples which we shall consider
  later comes from an abelian categories $ \CA = \coh(\CO_X)$ for some
  split superscheme $X$.
\end{example}
\begin{example}
  The category $\CD^{per}$ of perfect complexes on a scheme is a
  triangulated category. See \cite{Thomason1} for definitions.
\end{example}

\subsection{Triangular spectrum} \label{shv}

In this section we shall recall some definitions and results from
Balmer's papers \cite{PB1} and \cite{PB2}. Suppose $\CD$ is an essentially
small triangulated category
\begin{definition}
  A \emph{tensor triangulated category} is a triple $(\CD, \otimes,
  1)$ consisting of a triangulated category with symmetric monoidal
  bifunctor which is exact in each variable. The unit is denoted by
  $1$ (or $\Id$).
\end{definition}
\begin{definition}
  A \emph{thick tensor ideal} $\CA$ of $\CD$ is a full sub category
  containing $0$ and satisfying the following conditions:
  \begin{description}
  \item[(a)] $\CA$ is \emph{triangulated}: if any two terms of a
    distinguished triangle are in $\CA$ then third term is also in
    $\CA$. In particular direct sum of any two objects of $\CA$ is
    again in $\CA$ and this we refer as an {\bf additivity}.
  \item[(b)] $\CA$ is \emph{thick}: If $a \oplus b \in \CA \mbox{ then
    } a \in \CA$.
  \item[(c)] $\CA$ is \emph{tensor ideal}: if $a \text{ or } b \in
    \CA$ then $a \otimes b \in \CA$.
  \end{description}
\end{definition}
If $\mathcal{E} $ is any collection of $\CD$ then we shall denote by
$\langle \mathcal{E} \rangle$ the smallest thick tensor ideal
generated by this subset in $\CD$.

Now we shall give an explicit description of a thick tensor ideal generated by
some collection $\CE $ in a tensor triangulated category. We first use some
definitions from Bondal\cite{Bondal1} here. Recall $add(\CE)$ was defined as an
additive category generated by $\CE$ and closed under taking shifts inside
$\CD$. Similarly define $ideal(\CE)$ as a full sub category generated by objects
of the form $a \otimes x$ for each $ a \in \CD$ and $x \in \CE$. Therefore
$ideal(\CE)$ is closed under taking direct sum, shifts and tensoring with any
object of $\CD$. Recall the operation defined on sub categories i.e. $\CA \star
\CB$ is the full sub category generated by objects $x$ which fits in a
distinguished triangle of the form
        \[ 
           a \to x \to b \to a[1] \mbox{ with } a \in \CA \mbox{ and } b \in
\CB.
        \]
This was observed by Bondal\cite{Bondal1} et. al. that if $\CA$ and $\CB$ are
closed under shifts and direct sums then  $\CA \star \CB$ is also closed under
shifts and direct sums. Similarly we can see that if $\CA$ and $\CB$ are tensor
ideal then $\CA \star \CB$ is also tensor ideal. Take $smd(\CA)$ to be the full
subcatgory generated by all direct summands of objects of $\CA$. Now combining
these two operations we can define a new operation on collections of
subcategories as follows,
        \[ 
          \CA \diamond \CB := smd(\CA \star \CB).
        \]
Using this operation we can define the full subcategories $ \lan \CE \ran^n$ for
each non-negative integer as
        \[
           \lan \CE \ran^n := \lan \CE \ran^{n-1} \diamond \lan \CE \ran^0
\mbox{ where } \lan \CE \ran^0 := smd( ideal(\CE) ).
        \]
Now we can see following description of ideal generated by a collection $\CE$, 
\begin{lemma}
\label{ideal}
 $\lan \CE \ran = \cup_{n \geq 0} \lan \CE \ran^n$.
\end{lemma}
 Proof of the above lemma follows from the fact that right hand side subcategory
is a thick tensor ideal and contains every thick tensor ideal containing the
collection $\CE$.

\begin{definition}
  \begin{description}
  \item[(a)] An additive functor, $F : \CD_1 \to \CD_2$, is called an
    \emph{exact (or triangulated) } if it commutes with translation
    functor and takes distinguished triangle to a distingushed
    triangle.
  \item[(b)] \label{def:tensfunc} An exact functor, $F : \CD_1 \to
    \CD_2$, is called a \emph{tensor functor} if there exists a
    natural isomorphism $\eta(a, b) : F(a) \otimes F(b) \rightarrow
    F(a \otimes b)$ for objects $a$ and $b$ of $\CD_1$.
    % \item[(b)] An exact functor, $F : \CD_1 \to \CD_2$, is called a
    %   called \emph{ tensor functor} if $ F(A \otimes B) = F(A)
    %   \otimes F(B)$ for every objects $A ,B$ of $\CD_1$. Also we
    %   call a tensor functor $F$ to be \emph{unital tensor functor}
    %   if $F(1_{\CD_1}) = 1_{\CD_2}$.
  \item[(c)] A tensor functor, $F : \CD_1 \to \CD_2$, is called
    \emph{dominant} if $\langle F(\CD_1) \rangle = \CD_2$.
    % \item[(d)]\label{pseudo} Given an exact functor, $G : \CD_1 \to
    %   \CD_2$, is called \emph{pseudo tensor functor} (we shall not
    %   write this natural isomorphism for simplicity.) if there
    %   exists a natural isomorphism $\eta( a , b) : G(a) \otimes G(b)
    %   \to G(a \otimes b)$ for each objects $a, b$ of $\CD_1$.
    %   Further an
    %   additive functor, $G : \CD_1 \to \CD_2$, is called \emph{
    %     dense pseudo tensor functor} if it is pseudo tensor functor
    %   and
    %   $\langle G(\CD_1) \rangle = \CD_2$.
  \end{description}
\end{definition}
Note that every unital tensor functors is a dominant tensor
functor.

\begin{definition}
  A \emph{prime ideal} of $\CD$ is a proper thick tensor ideal $\CP
  \subsetneq \CD$ such that $a \otimes b \in \CP \implies a \in \CP
  \text{ or } b \in \CP$. And \emph{triangular spectrum} of $\CD$ is
  defined as set of all prime ideals, i.e.
  \[
  \Spc(\CD) = \{ \CP\ |\ \CP \text{ is a prime ideal of } \CD
  \}\text{.}
  \]
\end{definition}
The Zariski topology on this set is defined as follows: closed sets
are of the form
\[
Z(\CS) := \{\CP \in \Spc(\CD)\ |\ \CS \cap \CP = \varnothing\}
\text{,}
\]
where $\CS$ is a family of objects of $\CD$; or equivalently we can
define the open subsets to be of the form
\[
U(\CS) := \Spc(\CD) \backslash Z(\mathcal{S}) \text{.}
\]
In particular, we shall denote by
\[
\supp(a) := Z(\{a\}) = \{ \CP \in \Spc(\CD)\ |\ a \notin \CP \}
\text{,}
\]
the basic closed sets and similarly $U(\{a\})$ denotes the basic open
sets.

A collection of objects $\CS \subset \CD$ is called a \emph{tensor
  multiplicative family} of objects if $1 \in \CS$ and if $a, b \in
\CS \implies a \otimes b \in \CS$.

We shall recall here the following lemma(Lemma 2.2 in Balmer's
paper\cite{PB2}) which we shall need later,
\begin{lemma} \label{zorn} Let $\CD$ be a nontrivial tensor
  triangulated category and $\CI \subset \CD$ be a thick tensor
  ideal. Suppose $\CS \subset \CD$ is a tensor multiplicative family
  of objects s.t. $\CS \cap \CI = \varnothing$ Then there exists a
  prime ideal $\CP \in \Spc(\CD)$ such that $\CI \subset \CP \text{
    and } \CP \cap \CS = \varnothing$.
\end{lemma}
Balmer \cite{PB2} had also proved the functoriality of $\Spc$ on all
essentially small tensor triangulated category with a morphism given
by an unital tensor functors but it is not difficult to see that it is
also true for an essentially small tensor triangulated categories with
morphism given by a dominant tensor functor i.e. we have following
result,
\begin{proposition}
  Given $F : \CD_1 \to \CD_2 $ a dominant tensor functor, the map
  $\Spc(F) : \Spc(\CD_2) \to \Spc(\CD_1)$ defined as $\CP \mapsto
  F^{-1}(\CP)$ is well defined, continuous and for all objects $a \in
  \CD_1$, we have $\Spc(F)^{-1}(\supp(a)) = \supp(F(a))$ in
  $\Spc(\CD_2)$.

  This defines a contravariant functor $\Spc(-)$ from the category of
  essentially small tensor triangulated categories with dominant
  tensor functors as morphisms to the category of topological
  spaces. So if $F$, $G$ are two dominant tensor functors then $\Spc(G
  \circ F) = \Spc(F) \circ \Spc(G)$.
\end{proposition}
\begin{proof}
  (Similar to Balmer\cite{PB1})
\end{proof}
\begin{corollary}
  If a tensor functor $F : \CD_1 \to \CD_2$ is an equivalence then
  every quasi-inverse functor of $F$ is a dominant tensor functor. And
  also $\Spc(F)$ is a homeomorphism.
\end{corollary}
\begin{proof}
  First observe that the continuous map $\Spc(F)$ given by a dominant
  tensor functor is independent of natural isomorphism defining the
  tensor functor (recall definition \ref{def:tensfunc}). Now using
  functoriality of above proposition we have an homeomorphism whenever
  a quasi-inverse of $F$ is an dominant tensor functor. Suppose $G$ is
  a quasi-inverse of $F$. Since $G \circ F \simeq Id$, the exact
  functor $G$ is dominant. Suppose $\eta : F \circ G \to Id$ and $\mu
  :G \circ F \to Id$ are natural isomorphisms. Now we get a required
  natural isomorphism by composing as follows,
  \[
  G(a) \otimes G(b) \xrightarrow{\mu^{-1}} GF(G(a) \otimes G(b)) =
  G(FG(a) \otimes FG(b)) \xrightarrow{G(\eta_a \otimes \eta_b)} G(a
  \otimes b).
  \]
  Here we used a fact that $ G(\eta_a \otimes \eta_b)$ gives a natural
  transformation.
\end{proof}

Now we shall recall the definition of a structure sheaf defined on
$\Spc(\CD)$ as in Balmer \cite{PB2}.
\begin{definition}
  For any open set $U \subset \Spc(\CD)$, let $Z := \Spc(\CD)
  \setminus U$ be a closed complement and let $\CD_Z$ be the thick
  tensor ideal of $\CD$ supported on $Z$. We denote by $\CO_{\CD}$ the
  sheafification of following presheaf of rings: $U \mapsto
  \End(1_U)$ where $1_U \in \frac{\CD}{\CD_Z}$ is the image of the
  unit $1$ of $\CD$ via the localisation map.  And the restriction
  maps are defined using localisation maps in the obvious way.  The
  sheaf of commutative ring $\CO_{\CD}$ makes the topological space
  $\Spc(\CD)$ a ringed space, which we shall denote by $\Spec(\CD) :=
  (\Spc(\CD), \CO_{\CD})$.
\end{definition}
The following theorem was proved in Balmer\cite{PB2} which computes
the spectrum for certain tensor triangulated categories.
\begin{theorem}[Balmer]
  For $X$ a topologically noetherian scheme,
  \[
  \Spec(\CD^{per}(X)) \simeq X.
  \]
\end{theorem}

\subsection{$G$-sheaves} \label{G-shv} Throughout this section, $k$ is
field and $G$ be a finite group whose order is coprime to the
characteristic of $k$. Let $X$ be a smooth quasi-projective variety
over $k$, with an action of a finite group $G$ i.e. there is a group
homomorphism from $G$ to the automorphism group of algebraic variety
$X$. We say $G$ acts freely on $X$ if $gx \neq x$ for any $x \in
X$ and any $g \in G$ with $g \neq e$.  Recall following general result
proved in Mumford's Book \cite{Mumford} for the existence of group
quotient,
\begin{theorem}
  Let $X$ be an algebraic variety and $G$ a finite group of
  automorphisms of $X$.  Suppose that for any $x \in X$, the orbit
  $Gx$ of $x$ is contained in an affine open subset of $X$. Then there
  is a pair $(Y , \pi)$ where $Y$ is a variety and $\pi :X \to Y$ a
  morphism, satisfying:
  \begin{enumerate}
  \item \label{itm:gact} as a topological space, $(Y , \pi)$ is the
    quotient of $X$ for the $G$-action; and
  \item if $\pi_*(\CO_X)^G$ denotes the subsheaf of $G$-invariants of
    $\pi_*(\CO_X)$ for the action of $G$ on $\pi_*(\CO_X)$ deduced
    from \ref{itm:gact}, the natural homomorphism $\CO_Y \to
    \pi_*(\CO_X)^G$ is an isomorphism.
  \end{enumerate}
  The pair $(Y , \pi)$ is determined up to an isomorphism by these
  conditions.  The morphism $\pi$ is finite, surjective and separable.
  $Y$ is affine if $X$ is affine.

  If further $G$ acts freely on $X$, $\pi$ is an \'{e}tale morphism.
\end{theorem}
In the remark after the proof, Mumford further showed that
quasi-projective varieties always satisfies the hypothesis of above
theorem. We denote this quotient space (if it exists) as $X / G$. For
a variety $X$ with a $G$ action, and $H \subset G$ a subgroup, let
$X^H$ be the subvariety of fixed points of $H$.

% We can observe decomposition of variety as a $G$-space i.e. if $X$
% is any $G$-space then we can define a subset $X^H$ for any subgroup
% $H$ as a fixed point of all elements of $H$. we list some
% observations in following result,

\begin{proposition} \label{orbit decomp}
  With the notation in the above paragraph,
  \begin{enumerate}
  \item $X^H$ is a closed subvariety.
  \item If $H_1 \subseteq H_2$ are subgroups then we have a reverse
    inclusion $X^{H_2} \subseteq X^{H_1}$
  \item If $Y$ is any $G$-invariant component of $X$ then there exists
    an open subset of $Y$ with free action of $G/H$ for unique
    subgroup $H$.  A $G$-invariant component is defined to be a
    minimal $G$-invariant subset of $X$ with dimension equal to $\dim
    X$. Here dimension of an algebraic set is the maximum of the
    dimensions of its irreducible subsets.
  \item If $Y$ is any $G$-invariant algebraic subset of $X$ then there
    exists the set of subgroups $H_i$ for $i = 1, \ldots, r$ and open
    subsets $U_i, i = 1, \ldots, r$ s.t. $G/H_i$ acts freely on open
    subsets $U_i$ for $i = 1, \ldots, r$ of $Y$. Here $r$ is the number
    of $G$-invariant components of $Y$.  Also note that each open
    subsets $U_i$ for $i = 1, \ldots, r$ are pairwise disjoint.
  \end{enumerate}
\end{proposition}
\begin{proof}[Proof of 1.]
  Since $X^H = \cap_{h \in H} X^h$ where $X^h$ is a fixed points of
  automorphism corresponding to $h$ under the action. It is enough to
  prove that the invariant of any automorphism of a variety is a
  closed subset.  Since diagonal map gives a closed embedding we can
  take intersection with closed subset given by graph of
  automorphism. Hence it gives a closed subset of $X$.

  \noindent \emph{Proof of 2.} It clearly follows from the formulae
  $X^{H_i} = \cap_{h \in H_i} X^h$.

  \noindent \emph{Proof of 3.} Since for any algebraic subset there
  exists the subgroup $H$ s.t. $G/H$ acts faithfully(or
  effectively). Assume that $G$ acts faithfully on $Y$. Since for a
  faithful action $Y^H$ is a proper subset of $Y$ for any nontrivial
  subgroup of $G$. Define open subset of $Y$ as 
  \[
  U = Y - ( \cup_{H \trianglelefteq G} Y^H)
  \] where union on right side is over all nontrivial subgroups and
  now it is easy to see that $G$ acts freely on open set
  $U$.

  \noindent \emph{Proof of 4.} Using \emph{3.}, it is enough to prove
  that any algebraic subset can be uniquely written as union of
  $G$-invariant components of $X$. Since $X$ is noetherian it will be
  finite union of irreducible closed subsets. Take finite set $S$ of
  generic points of irreducible subsets of $X$ with same dimension as
  $X$. Now the action of $G$ on $X$ induces the action on finite set
  $S$ as an automorphism of $X$ will take any irreducible subset to
  another irreducible subset of same dimension. Thus $S$ can be
  uniquely written as a disjoint union of $G$-invariant subsets. By
  taking union of closure of these generic points in each invariant
  subsets, we get the $G$-invariant components of $X$. Clearly, any
  nonempty intersection of $U_i$ and $U_j$ for $i \neq j$ will give a
  proper $G$-invariant component, and this will contradict the
  minimality.
\end{proof}

We shall now look at some properties of $G$-sheaves (definition
\ref{def:gsheaf}).
%\begin{definition}
%  A $G$-sheaf (or $G$-equivariant sheaf or an equivariant sheaf with
%  respect to the group $G$) on $X$ is a sheaf $\CF$ with isomorphisms
%  $\rho_g : \CF \to g^*\CF$ for all $g \in G$ such that following
%  diagram
%  \[
%  \xymatrix{ \CF \ar[r]^{\rho_h} \ar[drr]_{\rho_{gh}} & h^*\CF
%    \ar[r]^{h^*\rho_g} & \ar@{=}[d] h^*g^*\CF
%    \\
%    & & (gh)^*\CF }
%  \]
%  is commutative for any pair $g, h \in G$. A $G$-sheaf is denoted by
%  a pair $(\CF , \rho )$.
%\end{definition}
Let the abelian category of all coherent $G$-sheaves be denoted by
$\coh^G(X)$.  In Tohoku paper of Grothendieck \cite{Grothendieck1} it
was proved that $\qcoh^G(X)$ has enough injectives. Therefore derived
functors of various functors like $\pi_* , \pi^*$ and $\otimes$ will
always exist similar to non-equivariant case and for simplicity we
shall use same notation.
% Further these derived functors will preseve coherence as in
% non-equivariant case.
We shall denote the bounded derived category of coherent $G$-sheaves
by $\CD^G(X)$. Similar to the case of $\CD^b(X)$ we have a symmetric
monoidal structure on $\CD^G(X)$ given by the (left) derived functor
of tensor structure on $\coh^G(X)$. Given an algebraic variety $X$
with an action of a finite group $G$ we have a natural morphism $\pi :
X \to Y$ which further gives a functor $\pi_* : \coh^G(X) \to
\coh^G(Y)$ and by taking $G$-invariant part of image we can define a
functor $\pi^G_* : \coh^G(X) \to \coh(Y)$ i.e. $\pi^G_*(\CF , \rho)
=(\pi_*(\CF , \rho))^G$ for all $(\CF , \rho) \in \coh^G(X)$.  We have
following result when $G$ acts freely on $X$, see Mumford's book
\cite{Mumford} for proof,

\begin{proposition}
  Let $\pi : X \to Y$ be a natural morphism given by free action of
  the finite group $G$ on $X$. The map $\pi^* : \coh(Y) \to \coh^G(X)$
  is an equivalence of abelian categories with the quasi-inverse
  $\pi^G_*$. Further locally free sheaves corresponds to locally free
  sheaves of the same rank.
\end{proposition}

Now we can extend above equivalence to get a tensor equivalence
$\pi^*$ between tensor triangulated categories $\CD^b(Y)$ and
$\CD^G(X)$.
% Note that the quasi-inverse of $\pi^*$ is just a tensor functor.
In general these two categories are not equivalent. If
we now take the case when $G$ acts trivially on an algebraic variety
$X$ then we have the canonical decomposition of each objects of
$\CD^G(X)$ similar to the case of finite dimensional representation of
finite group which is a particular case of this by taking $X$ to be a
single point $\Spec k$. So, more generally, we have following result,

\begin{proposition} \label{can_decomp} Let $X$ be an algebraic variety
  defined over $k$ with a $G$ action, and let $H$ be a subgroup of $G$
  with the property that it acts trivially on $X$. Then any object
  $(\CF, \rho) \in \CD^G(X)$ has the canonical decomposition as
  follows,
  \[
  (\CF, \rho) = \oplus_{\lambda} W_{\lambda} \otimes (\CF,
  \rho)_{\lambda}
  \]
  where $(\CF, \rho)_{\lambda} = (W^*_{\lambda} \otimes (\CF, \rho)
  )^H$ and $W_{\lambda}$ is a finite dimensional representation of the
  subgroup $H$ with the natural action of the group $G/H$ on $(\CF,
  \rho)_{\lambda}$.
\end{proposition}
\begin{proof}
  In any $k$-linear category we can identify $W \otimes \CG$, where
  $W$ is a finite dimensional vector space, with finite direct sum of
  object $\CG$.  Similarly we can define $Hom(V,\CG):= (V^* \otimes
  \CG)$ for any object $\CG$ and also we get the natural evaluation
  map $ev: V \otimes Hom(V \otimes \CG) \to \CG$. Moreover if $V$ and
  $\CG$ have $G$-action then $ev$ is $G$-equivariant map. In our
  situation using \emph{d\'{e}vissage} it is enough to study this $ev$
  map for pure sheaf, say $(\CF , \rho)$.  Now considering the induced
  action of $H$ on $(\CF , \rho)$ we get a map,
  \[
  ev :\oplus_{\lambda} W_{\lambda} \otimes (\CF ,\rho)_{\lambda} \to
  (\CF , \rho).
  \]
  We shall prove that the map $ev$ is an isomorphism. Since this is a
  local question we can reduce to the case of $A$-module $M$ with a
  $G$-action where $A$ is $k$-algebra. Since the map $ev$ is evidently
  a $A$-module morphism, it is enough to prove bijection of the map
  $ev$ as a $k$-linear map. Restriction of this map is an isomorphism
  for any $G$-invariant finite dimensional vector subspace of
  $M$. This follows from the canonical decomposition of finite
  dimensional representation of $G$. Since any element is contained in
  a finite dimensional $G$-invariant subspace of $M$ we get required
  bijection of the map $ev$. Since induced action of $H$ is trivial on
  $(\CF , \rho)_{\lambda} $ we have action of quotient group $G/H$.
\end{proof}

Note that if $G$ acts trivially on $X$ then we can take $H = G$ and as
a particular case we shall get the canonical decomposition,
\[
(\CF, \rho) = \oplus_{\lambda} V_{\lambda} \otimes (\CF,
\rho)_{\lambda}
\]
where $(\CF, \rho)_{\lambda} = (V^*_{\lambda} \otimes (\CF, \rho))^G$
and $V_{\lambda}$ is a finite dimensional representation of the group
$G$.  Now we shall give a distinguished triangle for any complex of
$G$-equivariant coherent sheaf $\CF$ over $X$. We have following
result,
\begin{proposition}  \label{supp red}
  Let $G$, $k$ and $X$ be as above.
  \begin{enumerate}
  \item Suppose $U$ is any $G$-invariant open subset of $X$ with
    induced action. If $\pi$ denotes the projector $\frac{1}{|G|}
    \sum_{g \in G} \rho_g^*$ on $X$ where $\rho_g$ is automorphism of
    $X$ coming from the action of $G$ then $ i_U^* \circ \pi = \pi
    \circ i_U^*$. Here, $\pi$ is also used to denote its restriction
    on open set $U$.
  \item Suppose $G$ acts faithfully on $X$. If $\CF \in \CD^G(X)$ with
    $\supph(\CF) = X$ then we have a distinguished triangle
    \[
    \pi^* \pi_*^G(\CF) \to \CF \to \CF_1
    \]
    with $\supph(\CF_1) \subsetneq \supph(\CF)$. Same is true if we
    have faithful action of $G$ on $\supph(\CF) \subsetneq X$.
  \end{enumerate}
\end{proposition}
\begin{proof}[Proof of 1.]
  Since $U$ is a $G$-invariant subset, each automorphism $\rho_g$ of
  $X$ induces an automorphism. For simplicity we use the same notation
  $\rho_g$. Now assertion immediately follows from the following
  commutative square,
  \[
  \xymatrix{ U \ar@{^{(}->}[r]^{i_U} \ar[d]^{\rho_g} & X
    \ar[d]^{\rho_g}\\
    U \ar@{^{(}->}[r]^{i_U} & X }
  \]
  and additivity.

  \noindent \emph{Proof of 2.} Since $G$ acts faithfully on $X$ we can
  use proposition \ref{orbit decomp} to get an open subset $U \subseteq
  X$ with free action of group $G$ . We shall use induction on
  amplitude length, $ampl(\CF)$.  When $ampl(\CF) = 1$ then $\CF$ is a
  shift of a coherent sheaf so enough to prove for coherent sheaf. Now
  using the fact that $\supph(\CF) = X$ we have $i_U^*(\CF) \neq
  0$. There is a natural morphism coming from adjunction and inclusion
  of $G$-invariant part, say $\eta : \pi^* \pi^G_*(\CF) \to
  \CF$. Using flat base change and part 1 of\ref{supp red} we get an
  isomorphism $i_U^* \pi^* \pi^G_*(\CF) \simeq \pi^*
  \pi^G_*(i_U^*\CF)$. Now this will give an isomorphism, as $G$ act
  freely on $U$, i.e. $i_U^*(\eta) : i_U^* \pi^* \pi^G_*(\CF) \to
  i_U^* \CF$ is an isomorphism. Hence cone of the map $\eta$ will have
  support outside an open set $U$. This completes the first step of
  induction.

  Now assume the for all $\CF$ with $ampl(\CG) \leq (n-1)$ we have
  such a distinguished triangle. Now consider $\CF$ with $ampl(\CF) =
  n$ with highest cohomology in degree $n$. We have usual truncation
  distinguished triangle $ \tau^{\leq(n-1)}(\CF) \to \CF \to
  \CH^n(\CF)[-n]$. Using exactness of $i_U^*$ and argument similar to
  first step of induction we have a following commutative diagram
  (we have used same notation $\eta$ for different sheaves),
  \[
  \xymatrix{ i_U^*\pi^*\pi^G_*\tau^{\leq(n-1)}(\CF) \ar[r]
    \ar[d]^{i_U^*(\eta)} & i_U^*\pi^*\pi^G_*\CF \ar[r]
    \ar[d]^{i_U^*(\eta)} &
    i_U^*\pi^*\pi^G_*\CH^n(\CF)[-n] \ar[d]^{i_U^*(\eta)}\\
    i_U^*\tau^{\leq(n-1)}(\CF) \ar[r] & i_U^*\CF \ar[r] &
    i_U^*\CH^n(\CF)[-n] }
  \]
  Since both the extreme vertical arrows are isomorphism using
  induction hypothesis, we have isomorphism of the middle
  $i_U^*(\eta)$. Therefore cone of the map $\eta$ will have proper
  support.
\end{proof}

\section{Example : Derived category of equivariant sheaves}\label{sec:dercatshv}
	In this section we shall compute Balmer's triangular spectrum for some
	particular examples. This computation of triangular spectrum also
motivates the need for some finer
	geometric structures attached to a given tensor triangulated category.
\subsection{Equivariant sheaves}
	In this example we shall compute Balmer's triangular spectrum for
	equivariant sheaves over some quasi-projective varieties with $G$-
	action. We shall always consider the varieties over some fixed field $k$.
        We shall first do some particular cases before going to general
	case. The general case will be done in the next subsection.

%	\begin{description} 
\subsubsection*{Case 1: $X$ is a point}
	Let $G$ be a finite group and $k$ be any field of char 0.
 %or coprime to order of group.
	As usual $\Rep(G)$ is the category of all finite dimensional $k$ linear
	representation of a group $G$. We can define a strict symmetric monoidal
	structure on this category  using the usual tensor product of
	representations i.e. if $V_1$ and $V_2$ are two representations of $G$
	then $V_1 \otimes V_2$ is the tensor product as $k$ vector spaces with
	diagonal action. We shall denote the bounded derived category of abelian
	category $\Rep(G)$ (resp. $\Rep(\{0\})$, for the trivial group $\{0\}$)
	as $\CD_{k[G]}$ (resp. $\CD_k$ ). We can extend the above tensor product
	of representations to get a symmetric tensor triangulated structure on
	$\CD_{k[G]}$.

	\begin{proposition}
 		$\Spec(\CD_{k[G]}) \cong \Spec(\CD_k) \cong \Spec(k)$. 
	\end{proposition}
	\begin{proof}
		Since $\Rep(\{0\})$ is a semisimple abelian category with $k$ as
its
		unit it is easy to see that $\Spec(\CD_k) \cong \Spec(k)$ as a
		variety. Therefore it is enough to prove the first isomorphism. 
The
		unit object of $\CD_{k[G]}$ is $k$ with endomorphism ring
isomorphic
		to $k$ so it remains to say that the trivial ideal, i.e. ideal
with
		only zero object, is the only prime ideal. This follows from the
		following lemma.

		\begin{lemma}
			Any representation of a finite group contains the
trivial
			representation as a direct summand of some tensor power.
		\end{lemma}
		\begin{proof}
			Let $V$ be a finite dimensional representation of a
finite group
			$G$. Since $\chr(k)$ is $0$ we have an graded vector
space
			isomorphism of the symmetric algebra with the symmetric
tensors
			contained in tensor algebra $T(V)$, i.e. subspace of
$T^n(V)$
			fixed by natural action of symmetric group $S_n$ for all
$n$.

			In fact, this isomorphism is given by a section of the
natural
			quotient map from $T(V)$ to $S(V)$ i.e. 
			\[
				v_1 \cdots v_k \mapsto
\frac{1}{k!}\underset{\sigma \in
				S_k}{\sum} v_{\sigma(1)} \otimes \ldots \otimes
				v_{\sigma(k)} \text{.}
			\]
			Now if we take the image of a nonzero element
$\underset{g \in
			G}{\prod}g.v \in S^{|G|}(V)$ under this isomorphism then
we
			shall get an nonzero fixed vector of $|G|$-symmetric
tensors and
			hence it will give the trivial representation as a
direct
			summand of $V^{\otimes|G|}$ using the semisimplicity of
			$\Rep(G)$.
		\end{proof}
		Using the above lemma and thickness it is easy to see that any
		non trivial ideal is full.
	\end{proof}
	But for the sake of generalisation  we shall give another proof of the
	first isomorphism.

	Consider the two exact tensor functors $F : \CD_{k[G]} \to \CD_k$ and
	$G : \CD_k \to \CD_{k[G]}$ where $F$ is the forgetful functor and $G$
	comes from the augmentation map of the group algebra $k[G]$ i.e. sending
	each complex of vector space to a complex of $k[G]$ module with the
	trivial action of a group $G$. Note that $F \circ G = Id$ and hence
	$\Spec(G) \circ \Spec(F) = \Id$. We now prove the following lemma.
	\begin{lemma}
		$\Spec(F) \circ \Spec(G) = \Id$.
	\end{lemma}
	\begin{proof}
		Let $\CP \in \Spec(\CD_{k[G]}) $ be a prime ideal. We want
		to prove that $(G \circ F)^{-1}(\CP) = \CP$. If $V \in
		\Mod(k[G])$ is any $k[G]$-module, then we have the canonical
		decomposition,
		\[
			V = \underset{\lambda}{\oplus} V_{\lambda} \otimes
			(V_{\lambda}^* \otimes V)^G
		\]
		where $V_{\lambda}$ is an irreducible representation of a group
$G$.
		Further $(V_{\lambda}^* \otimes V)^G$ is a direct summand of
		$(V_{\lambda}^* \otimes V)$ as is seen using the projector
		$\frac{1}{|G|} \sum_{g \in G} \rho_g$ where $\rho_g$ comes from
the
		action of a group $G$ on $(V_{\lambda}^* \otimes V)$. Since any
		complex in $\CD_{k[G]}$ is isomorphic to the direct sum of
		translates of the cohomology of that complex, to prove above
		assertion its enough to prove that $(G \circ F)^{-1}(\CP \cap
		\Mod(k[G])) = \CP \cap \Mod(k[G])$. Observe that,
		\beast
			V \in (\CP \cap \Mod(k[G])) & \Leftrightarrow &
(V_{\lambda}
				\otimes V)^G \in (\CP \cap \Mod(k[G])\\
			& & \text{ using thickness and additivity }\\
			&\Leftrightarrow & (V_{\lambda} \otimes V)^G \in (G
\circ
				F)^{-1}(\CP \cap \Mod(k[G]))\\
			& & \text{ Since } (G \circ F )(W) = W \text{ if } G
\text{ acts
				trivially on } W\\
			& \Leftrightarrow & V \in  (G \circ F)^{-1}(\CP \cap
				\Mod(k[G]))\\
			& & \text{ using thickness and additivity.}
		\eeast
		Therefore above observation completes the proof of lemma.
	\end{proof} 
	Hence using the above lemma we have another proof of the first
	isomorphism.

\subsubsection*{Case 2: $X$ smooth variety with a trivial $G$ action.}
	In this case, we shall extend the above example. Let $X$ be a smooth
	variety considered as a space with the trivial action of a finite group
	$G$.  Recall the definitions and some properties of a $G$-sheaves from
	the preliminary section \ref{basics}. Let $\coh(X)$ (resp. $\coh^G(X)$)
be
	the abelian category of all coherent sheaves (resp. coherent
$G$-sheaves)
	over $X$. We have  two functors $F$ and $G$ similar to the previous
	example defined as follows,
	\beast
		F : \coh^G(X)  \to  \coh(X)  & \& & G : \coh(X)  \to  \coh^G(X)
\\
		(\CF, \rho)  \mapsto  \CF ~~~~~~~~ &  & ~~~~~~~~~~~~ \CF
		\mapsto  (\CF, id)
	\eeast
	Note that the functor $F$ (respectively $G$) is a faithful (respectively
	fully faithful) exact functor. Thus we get two exact derived
	functors of the above two functors, $F : \CD^G(X) \to \CD^b(X)$ and $G :
	\CD^b(X) \to \CD^G(X)$ which by abuse of notation are denoted by the same
	symbols.

	Recall that $\CD^G(X)$ and $\CD^b(X)$ are a tensor triangulated
	categories which makes the functors $F$ and $G$ unital tensor
	functors and hence using the functorial property of ``Spec'' we shall
	get two morphisms $\Spec(F) : \Spec(\CD^b(X)) \to \Spec(\CD^G(X))$ and
	$\Spec(G) :  \Spec(\CD^G(X)) \to \Spec(\CD^b(X))$. Now we have
	following result,
	\begin{proposition}
		$\Spec(\CD^G(X)) \cong \Spec(\CD^b(X)) \cong X$.
	\end{proposition}
	\begin{proof}
		Here, the second isomorphism was proved by Balmer \cite{PB2}
which
		enables him to reconstruct the variety from its associated
tensor
		triangulated category of coherent sheaves. We shall use the idea
of
		previous example to prove the first isomorphism. 
		
		Since $F \circ G = Id$, functoriality of the ``Spec'' will give
		$\Spec(G) \circ \Spec(F) = \Id$. Now it remains to prove that
		$\Spec(F) \circ \Spec(G) = \Id$. Note that every object $(\CF,
\rho)
		\in \CD^G(X)$ has the canonical decomposition as follows,
		\[
			(\CF, \rho) = \underset{\lambda}{\oplus} V_{\lambda}
\otimes
			(\CF , \rho)_{\lambda}
		\]
		where $(\CF, \rho)_{\lambda} = (V^*_{\lambda} \otimes (\CF,
		\rho))^G$ and $V_{\lambda}$ is a finite dimensional irreducible
		representation of the group $G$, see section \ref{basics} for
proof.
		Also note that $(\CF ,\rho)_{\lambda}$ is an ordinary sheaf with
		the trivial action of a group $G$ and also using similar
projector
		as above, i.e.  $\frac{1}{|G|} \underset{g \in G}{\sum} \rho_g$,
we
		can prove that $(\CF ,\rho)_{\lambda}$ is an direct summand of
the
		sheaf $(V^*_{\lambda} \otimes (\CF , \rho) )$.  Now we use the
		following lemma.
		\begin{lemma}
			$\Spec(F) \circ \Spec(G) = \Id$.
		\end{lemma}
		\begin{proof}
			Let $\CP \in \Spec(\CD^G(X))$ be a prime ideal. We want
to prove
			that $(G \circ F)^{-1}(\CP) = \CP$. Now using the
canonical
			decomposition of each objects of the triangulated
category
			$\CD^G(X)$. we have,
			\beast
				(\CF , \rho) \in \CP  & \Leftrightarrow & (\CF ,
					\rho)_{\lambda} \in \CP  \text{ using
thickness,
					additivity and projector}\\
				& \Leftrightarrow & (\CF , \rho)_{\lambda} \in
(G \circ
					F)^{-1}(\CP)\\
				& & \text{Since } (G \circ F )(\CF, id) = (\CF,
id) \text{
					if } G \text{ acts trivially i.e. } \rho
= id\\
				& \Leftrightarrow & (\CF , \rho) \in  (G \circ
F)^{-1}(\CP )
					\\
				& & \text{ using thickness, additivity and
projector.} 
			\eeast
			Hence the above observation completes the proof of
lemma.
		\end{proof}
		Now, using the above lemma, it follows that $\Spec(F)$ is an
		isomorphism between $\Spec(\CD^G(X))$ and $\Spec(\CD^b(X))$.
	\end{proof}
\begin{remark}
\begin{enumerate}
\item In case 1 the second proof also works for fields with characteristic co-prime to
 order of the group $G$.
\item The proof for the case of trivial action on smooth varieties doesn't need assumption
of quasi-projectivity on the variety $X$ which is used later for the general case for
 the existence of group quotients.  
\end{enumerate}
\end{remark}

\subsubsection*{Case 3: $G$ acts freely on a smooth variety $X$}
	Now we shall consider the case where a finite group $G$ acts freely on
	$X$. We refer to section \ref{basics} for the definition. Recall that we
	have a canonical map $\pi : X \to Y := X / G$ which is a $G$-equivariant
	map with the trivial action of $G$ on $Y$. Now we can also define two
	functors associated with $\pi$: $\pi^*: \coh(Y) \to \coh^G(X)$ and
	$\pi^G_* : \coh^G(X) \to \coh(Y)$ where $\pi^G_* = G\text{-equivariant
	part of } \pi_*$. We had also seen in \ref{basics} that $\pi^*$ is a
	tensor functor in general; and when $G$ acts freely it is also an
	equivalence of categories with $\pi^G_*$ as its quasi-inverse. Hence we
	shall get an equivalence of the tensor triangulated categories $\CD^b(Y)$
	and $\CD^G(X)$. Since an equivalence gives an isomorphism of ``Spec'',
	(cf. \ref{basics}), therefore we get an isomorphism $\Spec(\pi^*):
	\Spec(\CD^G(X)) \to \Spec(\CD^b(Y))$ with its inverse given by
	$\Spec(\pi^G_*)$. In fact using case 2 and this argument,
	we can give slightly more general statement as follows.
	\begin{corollary}
		Suppose finite group $G$ acts freely on a quasi-projective
variety
		$X$ modulo some normal subgroup $H$. In other words, the
subgroup
		$H$ acts trivially, and the induced action of the quotient group
		$G/H$ is free. Then
		\[
			\Spec(\CD^G(X)) \cong \Spec(\CD^b(Y)) \cong Y
		\]
		where $Y := X / G$ as before.
	\end{corollary}
	\begin{proof}
		As mentioned above, the proof goes in similar lines as in case
2,
		using a more general canonical decomposition of objects of
		$\CD^G(X)$: 
		\[
			(\CF , \rho) = \underset{\lambda}{\oplus} W_{\lambda}
\otimes
			(\CF, \rho)_{\lambda}
		\]
		where $(\CF, \rho)_{\lambda} = (W^*_{\lambda} \otimes (\CF ,
		\rho))^H$, $W_{\lambda}$ is a finite dimensional irreducible
		representation of the group $H$, and the group $G/H$ acts
naturally
		on $(\CF , \rho)_{\lambda}$. See section \ref{basics} for the
proof.
	\end{proof}
	
	Finally we tackle the general case. Since the proof is a bit long
	involving some steps we devote a full subsection to it.

\subsection{Case 4: The general case}
	Finally in this case we shall consider the more general situation of a
	finite group $G$ acting on a smooth quasi-projective variety $X$ and we
further
	assume that the group $G$ acts faithfully. Define $\pi : X \to Y := X /
	G$ as above an $G$-equivariant map when considered with the trivial
	action of a group $G$ on $Y$. Note that for a finite group the quotient
	space always exists \ref{basics}. We have a following main result,
	\begin{proposition}
		$\Spec(\CD^G(X)) \cong \Spec(\CD^{per}(Y) \cong Y$.
	\end{proposition}

	Here again as before the second isomorphism is a particular case of the
	more general reconstruction result of Balmer \cite{PB1} \cite{PB2}.
	Hence we shall just prove the first isomorphism. We know there are two
	exact functors $\pi^* : \CD^{per}(Y) \to \CD^G(X)$ and $\pi_* : \CD^G(X) \to
	\CD^{per}(Y)$. We also know that the map $\pi^*$ is an unital tensor functor
	and hence it will give the map $\Spec(\pi^*): \Spec(\CD^G(X)) \to
	\Spec(\CD^{per}(Y))$. Note that $\pi_*$ need not be a tensor functor.
	We shall prove that $\Spec(\pi^*)$ is a closed bijection and induces an
	isomorphism for the structure sheaves. To simplify the proof we will
	break it in several steps.

\subsubsection*{Step 1: $\Spec(\pi^*)$ is onto.}
	Suppose $\FQ \in \Spec(\CD^{per}(Y))$ is a prime ideal then we want to
	construct an prime ideal $\FP$ in $\Spec(\CD^G(X))$ such that $\FQ =
	(\pi^*)^{-1}(\FP)$. Recall that $\lan \pi^*(\FQ) \ran$ denotes the thick
	tensor ideal generated by the image of $\FQ$ via functor $\pi^*$ in a
	tensor triangulated category $\CD^G(X)$. We have a following lemma which
	uses the explicit description of thick tensor ideal $\lan \pi^*(\FQ)
	\ran$.
	\begin{lemma}
		$\pi_*(\lan \pi^*(\FQ) \ran) \subseteq \FQ$.
	\end{lemma}
	\begin{proof}
		To prove this lemma, we use lemma\ref{ideal} i.e.
		\[
			\lan \pi^*(\FQ) \ran = \cup_{n \geq 0} \lan \pi^*(\FQ)
\ran^n
		\]
		where $\lan \pi^*(\FQ) \ran^n$ constructed inductively by taking
		$\lan \pi^*(\FQ) \ran^0$ as the summands of tensor ideal
generated by
		$\pi^*(\FQ)$ and $\lan \pi^*(\FQ) \ran^n$ to be the thick tensor
ideal
		containing cone of morphism between any two objects of $\lan
\pi^*(\FQ)
		\ran^{(n-1)}$ and$\lan \pi^*(\FQ)
		\ran^0$ . Here cone of a morphism refers to the third object of
any
			distinguised triangle having this morphism as a base
                or equivalently we can use $\diamond$ operation. The above
equality follows from the lemma\ref{ideal} proved
		earlier.

		We shall use induction on $n$ in the above explicit description.
For
		$n = 0$, given $\CF \in \FQ$,
		\[
			\pi_*(\pi^*(\CF)\otimes \CG) = \CF \otimes \pi_*(\CG)
\in \FQ
			\text{,}
		\]
		and hence $\pi_*(\lan \pi^*(\FQ) \ran^0) \subseteq \FQ$ using
thickness of $\FQ$. 
		
		Using induction suppose we know that $\pi_*(\lan \pi^*(\FQ)
		\ran^{(n-1)}) \subseteq \FQ$. Since $\pi_*$ is an exact functor,
it
		follows that the image under $\pi_*$ of a cone of any morphism
is
		a cone of $\pi_*$ of the morphism. Hence using the
		triangulated ideal property and thickness of $\FQ$ it follows
that $\pi_*(\lan
		\pi^*(\FQ) \ran^n) \subseteq \FQ$. Therefore we have 
$\pi_*(\lan
		\pi^*(\FQ) \ran) = \pi_*(\cup_{n \geq 0} \lan \pi^*(\FQ) \ran^n)
		\subseteq \FQ$.
	\end{proof}

	\begin{lemma}
		$\pi^*(\CD^{per}(Y) \setminus \FQ) \cap \lan \pi^*(\FQ) \ran =
		\varnothing$.
	\end{lemma}
	\begin{proof}
		To prove this by contradiction, suppose that there exists an
object
		$\CG \in (\CD^{per}(Y) \setminus \FQ)$ such that $\pi^*(\CG) \in \lan
		\pi^*(\FQ) \ran$. Then using the above lemma $\pi_*(\pi^*\CG)
\in
		\FQ$. On the other hand, the projection formula implies
		$\pi_*(\pi^*\CG) = \CG \otimes \pi_*(\CO_X)$, which we saw is in
		$\FQ$.
		
		Using the primality of $\FQ$ it follows that $ \pi_*(\CO_X) \in
		\FQ$.  Now $(\pi_*(\CO_X))^G = \CO_Y$ is a direct summand of
		$\pi_*(\CO_X)$ by the canonical decomposition of a $G$-sheaves
on
		$Y$.  Hence $\CO_Y$ is an object of $\FQ$; which is absurd.
	\end{proof}

	To complete Step 1, we apply Balmer's result \ref{zorn} to get an prime
	ideal $\FP$, such that $\pi^*(\CD^{per}(Y) \setminus \FQ) \cap \FP =
	\varnothing$ and $\lan \pi^*(\FQ) \ran \subseteq \FP$. Hence we shall
get $\FQ = (\pi^*)^{-1}(\FP)$ which proves
	the surjectivity of the map $\Spec(\pi^*)$.

\subsubsection*{Step 2: Injectivity of $\Spec(\pi^*)$} 
	First we shall give  proof for the case of a \emph{smooth projective
	curve} as it is simpler than the general case. We have a following basic
	result for the case of a smooth projective curve which we shall use in
	the proof.
	\begin{proposition} \label{curve}
		\begin{enumerate}
			\item
				Any object of $\CD^b(\CA)$,for a hereditary
abelian category
				$\CA$, is noncanonically isomorphic to the
direct sum of its
				cohomologies with shifts. In particular, this is
true for
				$\CA = \coh(X)$ where $X$ is a smooth projective
curve.
			\item
				Every coherent sheaf over smooth projective
curve $X$ is a
				direct sum of a coherent skyscraper sheaves and
a locally
				free coherent sheaves. 
		\end{enumerate}
	\end{proposition}
	Using above result we prove the following proposition.
	\begin{proposition}
		The map $\Spec(\pi^*): \Spec(\CD^G(X)) \to \Spec(\CD^b(Y))$ is an
		injective map between smooth projective curves $X$ and $Y$.
	\end{proposition}
	\begin{proof}
		Suppose not, let $\FP_1$, $\FP_2$ be two distinct points of
		$\Spec(\CD^G(X))$ mapping to the same point $\FQ_y$ where $y$ is
		given by the identification of $\Spec(\CD^b(Y))$ with $Y$. Let
$\CF$ be
		an element of $\FP_1$ and using the above proposition
(\ref{curve}) we
		can assume that it is a pure sheaf. We have the following lemma
which
		gives a restriction on the homological support of such elements.
		\begin{lemma}
			$\supp(\CF) \subseteq (X \setminus \pi^{-1}(y))$.
		\end{lemma}
		First let us complete the proof of the proposition assuming this
		lemma.  From the lemma it follows that $\supp(\CF)$ is a proper
		subset of $X$ with a $G$-action. Therefore $\supp(\CF$) is a
finite
		set of points and using thickness further we can assume that it
is a
		single orbit. Suppose $H$ is a stabiliser of this orbit. Then
$G/H$
		will act freely on supp($\CF$). Now we have the decomposition,
		\[
			\CF = \oplus_{\lambda} W_{\lambda} \otimes \CF_{\lambda}
\simeq
			\oplus_{\lambda} W_{\lambda} \otimes \pi^* \pi^{G/H}_*
			(\CF_{\lambda})
		\]
		where $W_{\lambda}$ is an irreducible representation of $H$.
		Therefore $\CF \in \FP_1 \cap \FP_2$, since using a projector
		$\CF_{\lambda} = (W_{\lambda}^* \otimes \CF)^G \simeq
		\pi^*\pi^{G/H}_*(\CF_{\lambda}) \in \FP_1 \cap \FP_2 $, and
hence
		$\FP_1 \subseteq \FP_2$. Using similar arguments we can prove
$\FP_2
		\subseteq \FP_1$. This is a contradiction as $\FP_1$ and $\FP_2$
are
		distinct points.

		This proves the proposition assuming the lemma. Next we prove
the
		lemma.
	\end{proof}
	\begin{proof}[Proof of lemma]
		We prove it by contradiction. Assume $\supp(\CF) \cap
\pi^{-1}(y)
		\neq \varnothing$. If $y$ is a closed point then we can assume
that
		$\supp(\CF) = \pi^{-1}(y)$ since we can always tensor with the
		object $\CO_{\pi^{-1}(y)}$ which will give an object of $\FP_1$.
And
		if $H$ is a stabiliser of this finite $G$ set then we shall have
the
		usual decomposition $\CF = \oplus_{\lambda} W_{\lambda} \otimes
		\pi^*\pi^{G/H}_*(\CF_{\lambda})$. Hence we shall get an object,
		$\pi^{G/H}_*(\CF_{\lambda})$ , of $\FQ_y$ supported on $y$ which
is
		a contradiction.
		
		Similarly, if $y$ is a generic point of $X$ then using the above
		proposition \ref{curve} we can assume that $\CF$ is a
		$G$-equivariant vector bundle. Now using a short exact sequence,
		inspired from short exact sequence (4.7) from
paper\cite{ramanan}, $
		0 \to \pi^*\pi^G_*(\CF) \to \CF \to \CF'\to 0$ with $\CF'$
supported
		on a points, we can prove $\CF'\in \FP_1$ and hence
		$\pi^*\pi^G_*(\CF) \in \FP_1$. Now using our assumption $
		\pi^G_*(\CF) \in \FQ_y$ which is a contradiction as
$\pi^G_*(\CF)$
		is a vector bundle.
	\end{proof}

	This finishes the case of curves. For the general case we need the
	following propostion.
	\begin{proposition}
		The map $\Spec(\pi^*): \Spec(\CD^G(X)) \to \Spec(\CD^{per}(Y))$ is an
		injective map where $X$ is a smooth quasi-projective
		varieties of dimension $n$.
	\end{proposition}
	\begin{proof}
		Suppose not, let $\FP_1$, $\FP_2$ be two distinct points of
		$\Spec(\CD^G(X))$ which maps to the same point $\FQ_y$ i.e.
		$(\pi^*)^{-1}(\FP_1) = (\pi^*)^{-1}(\FP_2) = \FQ_y$. Let $\CF
\in
		\FP_1$ be an complex of $G$-equivariant sheaves. We need
following
		lemma.
		\begin{lemma} \label{tower}
			\begin{enumerate}
				\item
					There exists a tower of distinguished
triangles for each
					element $\CF$,
					\begin{displaymath}
						\xymatrix@C=10pt@M=2pt{
							\CF=\CF_0 \ar[rr] & &
\CF_1 \ar@{-->}[dl] &
								\cdots &
\CF_{m-1} \ar[rr]& & \CF_m
								\ar@{-->}[dl]
\ar@{=}[dr]\\
							& \CG_1 \ar[ul] & &
\cdots& &\CG_{m-1} \ar[ul] &
								& \CG_m
						}
					\end{displaymath}
					where $\CG_i =
\underset{\lambda_i}{\bigoplus}
					W_{\lambda_i} \otimes \pi^*\pi^{G/H_i}_*
					(\CF_{\lambda_i})$ with the 
					sum being over the irreducible
representations of the
					corresponding $H_i$'s, $\supph(\CF_m)
\subsetneq \ldots
					\subsetneq \supph(\CF)$ and
					$\supph(\pi^{G/H_j}_*(\CF_{\lambda_j}))
\subseteq
					\pi(\supph(\CF_{\lambda_j} )) =
					\supph(\pi_*(\CF_{\lambda_j}))$.
				\item
					$\supph(\CF) \subseteq (X \setminus
\pi^{-1}(y))$.
			\end{enumerate}
			Also note, we can prove that any $\CF \in \CD^G(X)$ with
the
			homological support contained in $(X \setminus
\pi^{-1}(y))$
			will be an element of $\FP_1$.
		\end{lemma}
		We shall first complete the proof of the proposition assuming
this
		lemma. By the lemma, the homological support of $\CF$ is a
proper
		closed subset of $X$ not containing $\pi^{-1}(y)$. Note that in
		above lemma similar to the case of $\CF$'s we also have
decreasing
		filtration of a homological supports for $\CG$'s. If we start
with
		$\CF \in \FP_1$ then using (2.) of the lemma we have
		$\supph(\pi^{G/H_1}_*(\CF_{\lambda_1})) \subseteq (Y \setminus
y)$
		and hence $\CG_1 \in \pi^*(\FQ_y) \subseteq \FP_1 \cap \FP_2$.
		Therefore by definition of the prime ideal $\CF_1 \in \FP_1$.
Again
		using lemma we have the restriction on homological  support of
		$\CF_1$ which gives  $\CG_2 \in  \pi^*(\FQ_y) \subseteq \FP_1
\cap
		\FP_2$.  Now continuing like this we can prove that   $\CG_j \in
		\pi^*(\FQ_y) \subseteq \FP_1 \cap \FP_2$ for $j = 1, \ldots, m$.
 In
		particular, $\CF_m \in \FP_2$ and hence $\CF_{m-1} \in \FP_2$.
Now
		continuing in the reverse direction we can prove that $\CF \in
		\FP_2$. Thus $\FP_1 \subseteq \FP_2$. Similarly by symmetry we
can
		prove that  $\FP_2 \subseteq \FP_1$. This is a contradiction as
		$\FP_1$ and $\FP_2$ are distinct points.
	\end{proof}

	\begin{proof}[Proof of the lemma.]
		\emph{Proof of 1.} To prove the first part we consider the
		homological support of $\CF$ as $G$-subset  of the $G$-set $X$
and induct on dimension of it. If dimension is zero then it will be set of
$G$-invariant points and we shall get the direct sums of skyscrapers on these
points. If we have free action of $G/H$ for some subgroup $H$ then we shall have
the canonical decomposition\ref{can_decomp}. This will prove first step of
induction. Assume now for dimension of $\supph(\CF) \leq (n-1)$ we have a tower
as in statement of lemma. Now consider $\CF$ with dimension of $\supph(\CF) =
n$. Here
		$\supph(\CF)$ is a union of $G$-invariant components and using
the proposition\ref{orbit decomp} we shall get the open subsets $U_i$ for $ i
=1, \ldots,k$ and subgroups $H_i$ for $i=1,\ldots,k$. As observed before these
open sets are mutually disjoint and there is a free action of group $G/H_i$ for
$i=1,\ldots,k$ on each $U_i$ respectively. We can decompose $i_U^*(\CF) =
\oplus_{i=1}^k i_{U_j}^*(\CF) \mbox{ for } j =1,\ldots,k$.   Start with an open
subset $U_1$. We have the decomposition \ref{supp red} of $i_{U_1}^*(\CF)$,
		\[
                        i_{U_1}^*(\CF) = \underset{\lambda_1}{\oplus}
W_{\lambda_1} \otimes
			\CF_{\lambda_1}
		\]
		where $W_{\lambda_j}$ is an irreducible representation of
subgroup
		$H_1$. In this decomposition, all the $\CF_{\lambda_1}$ are also
		$G/H_1$-sheaves over open subset $U_1$. Using adjunction and the
inclusion we get a canonical isomorphism\ref{supp red} $\eta :
\pi^*\pi^{G/H_1}_*(\CF_{\lambda_1}) \to
		\CF_{\lambda_1}$ over open set $U_1$. This will give an
isomorphism $ \eta : \pi ^* \pi_*^{G/H_1}(i_{U_1}^* \CF) \to i_{U_1}^*\CF$ of
objects over open set $U_1$ by using additivity. Now this isomorphism gives an
isomorphism $i_{U_1}^*(\eta) :  i_{U_1}^*(\pi^*\pi^{G/H_1}_*(\CF_{\lambda_1}))
\to
		i_{U_1}^*(\CF_{\lambda_1})$. This follows from
		flat base change and some functorial properties. For the
		diagram,
		\[
			\xymatrix{
			U_1 \ar@{^{(}->}[r]^{i_{U_1}} \ar[d]^{\pi} & X
\ar[d]^{\pi}\\
			V_1 \ar@{^{(}->}[r]^{i_{V_1}} & Y
			}
		\]
		we have canonical isomorphisms,
		\begin{align*}
			i_{U_1}^*(\pi^*\pi^{G/H_1}_*(\CF_{\lambda_1})) &\simeq
\pi^*i^*_{V_1}
			(\pi_*(\CF_{\lambda_1}))^{G/H_1} \simeq \pi^*(i_{V_1}^*
			\pi_*(\CF_{\lambda_1}))^{G/H_1}\\
			&\simeq \pi^*(\pi_*i^*_{U_1}(\CF_{\lambda_1}))^{G/H_1} =
			\pi^*\pi_*^{G/H_1}i^*_{U_1}(\CF_{\lambda_1}) \text{,}
		\end{align*}
		which will imply that the map $i^*_{U_1}(\eta)$ is an
isomorphism.
		Therefore the cone of a map $\eta$, say $\CF_1$, will have the
		property that $i^*_{U_1}(\CF_1) = 0$ and hence $\supph(\CF_1)
\subseteq
		(X \setminus U) \subsetneq \supph (\CF)$. And since $\pi$ is an
		affine map then $\pi_*$ will be exact and hence we can prove
that
		$\pi(\supph(\CF_{\lambda_1} )) =
		\supph(\pi_*(\CF_{\lambda_1}))\supseteq
		\supph(\pi^{G/H_1}_*(\CF_{\lambda_1})$. Now we can proceed
similarly
		with $\CF_1$ which has less number of $G$-invariant components
than
		$\CF$ and hence in finitely many steps (in less than $k$ steps)
dimension of homological
		support will drop. Hence we shall get $\CF_i$ and $\CG_i$ for
$i=1,\ldots,l$ with the stated restrictions on supports. The dimension of
$\supph(\CF_l) \leq (n-1)$ and hence using induction we are done.

		\emph{Proof of 2.} Suppose $\supph(\CF) \cap \pi^{-1}(y) \neq
		\varnothing$ and hence we get $\CF' = \CF \otimes
		\CO_{\pi^{-1}(\bar{y})} \in \FP_1$. Observe that $\supph(\CF') =
		\pi^{-1}(\bar{y}) = \overline{\pi^{-1}(y)}$. Now applying the
same
		procedure as in 1., we shall get a distinguished triangle
		\[
			\underset{\lambda}{\oplus} W_{\lambda} \otimes
			\pi^*\pi^{G/H}_*(\CF'_{\lambda}) \to \CF' \to \CF'' \to
		\]
		with $\supph(\CF'') \subsetneq \supph(\CF')$ and hence again
		applying 1., we can prove that $\CF'' \in
\langle\pi^*(\FQ_y)\rangle
		\subseteq \FP_1$. But this gives
$\pi^*\pi^{G/H}_*(\CF'_{\lambda})
		\in \FP_1$ with $\supph(\pi^{G/H}_*(\CF'_{\lambda})) = \bar{y}$
%(needs more justification for above equality)
%????????????????????????
		which is a contradiction as $\pi^{G/H}_*(\CF'_{\lambda}) \notin
		\FQ_y$.
	\end{proof}

\subsubsection*{Step 3: $\Spec(\pi^*)$ is closed and hence is a
homeomorphism.}
	Here we need  bijection of the above step to prove closedness of the map
	$\Spec(\pi^*$). We shall use the fact that $W \otimes \CO_X \notin \FP$
	for any finite dimensional representation and any  prime ideal $\FP$.
	Indeed this follows from the fact that every representation contains the
	trivial representation as a direct summand of some tensor power i.e.
	$\CO_X \subseteq (V_{\lambda} \otimes \CO_X)^{\otimes |G|}$. Since
	$\supp(a)$, $a \in \CD^G(X)$, are the basic closed sets therefore it is
	enough to prove that image under the map $\Spec(\pi^*$) are closed. Now
	to prove this we shall use the description given in lemma \ref{tower}
	for any object of $\CD^G(X)$. Letting $b_{\lambda_j} =
	\pi^{G/H_j}_*(a_{\lambda_j})$ for simplicity, we have the following
	lemma.
	\begin{lemma}
		%$\supp(a) = \underset{j}{\cup}~\underset{\lambda_j}{\cup}
		%\supp(b_{\lambda_j})$. 
		$\Spec(\pi^*)(\supp(a)) = \bigcup_j\bigcup_{\lambda_j}
\supp(b_{\lambda_j})$. 
	\end{lemma}
	\begin{proof}
		Given $a \in \FP$ we have $b_{\lambda_j}$'s as in lemma
\ref{tower}.
		Now,
		\beast
			a \in \FP & \Leftrightarrow & W_{\lambda_j} \otimes
				\pi^*(b_{\lambda_j}) \in \FP ~~~~~~  \forall j ,
\lambda_j\\
			& \Leftrightarrow &  \pi^*(b_{\lambda_j}) \in \FP
\text{,
				~~~~since } ~W_{\lambda_j} \otimes \CO_X \notin
\FP. \\
			\text{ Therefore}~~~~~~~~~~ a \notin \FP &
\Leftrightarrow &
				\exists~ \lambda_j \text{~~s.t.~~}
\pi^*(b_{\lambda_j})
				\notin \FP.
		\eeast
		Let $\FP \in \supp(a)$ and hence by the definition $a \notin
\FP$.
		Now using the above observation there exists a $\lambda_j$ such
that
		$\pi^*(b_{\lambda_j}) \notin \FP$ i.e. $b_{\lambda_j} \notin
		(\pi^*)^{-1}(\FP) = \Spec(\pi^*)(\FP)$ and hence
$\Spec(\pi^*)(\FP)
		\in \supp(b_{\lambda_j})$. Therefore $\Spec(\pi^*)(\supp(a))
\subseteq
		\cup_j \cup_{\lambda_j} \supp(b_{\lambda_j})$.
		
		Conversely suppose $\FQ \in \cup_j\cup_{\lambda_j}
		\supp(b_{\lambda_j})$ and hence $\FQ \in \supp(b_{\lambda_j})$
for
		some $\lambda_j$.  Therefore by definition $b_{\lambda_j} \notin
		\FQ$ but using the bijection of the map $\Spec(\pi^*)$ we have
		$b_{\lambda_j} \notin (\pi^*)^{-1}(\FP) = \FQ$ for some $\FP$.
Now
		it follows that $ \pi^*(b_{\lambda_j}) \notin \FP$ and once
again
		using the above observation we have $a \notin \FP$ i.e. $\FP \in
		\supp(a)$. Hence we have $\cup_j\cup_{\lambda_j}
		\supp(b_{\lambda_j}) \subseteq \Spec(\pi^*)(\supp(a))$.
	\end{proof}
	Since union in right hand side of above lemma is finite it follows that
	the image of $\supp(a)$ under the map $\Spec(\pi^*)$ is closed for all
	$a \in \CD^G(X)$. Hence the map $\Spec(\pi^*)$ is a closed map and
	therefore it is a homeomorphism.

\subsubsection*{Step 4: $\Spec(\pi^*)$ is an isomorphism.}
	In this step we shall prove that the above homeomorphism spec($\pi^*$)
	is, in fact, an isomorphism. We begin by proving the following lemma
	which we shall use later.
	\begin{lemma} \label{adjunction}
		There exist a natural transformation $\eta : \pi^* \pi^G_* \to
Id$
		(resp. $\mu : Id \to \pi^G_* \pi^*$) such that $\eta(\CO_X) =
id$
		(resp.  $\mu(\CO_Y) = id$) where $\pi^* \pi^G_*(\CO_X) = \CO_X$
		(resp. $\pi^G_* \pi^*(\CO_Y) = \CO_Y $).
	\end{lemma}
	\begin{proof}
		We shall prove the existence of $\eta$, as $\mu$ can be found
using
		similar arguments. Since the functor $\pi^*$ is a left adjoint
of
		the functor $\pi_*$ we have a natural transformation $\eta' :
\pi^*
		\pi_* \to \Id$ given by the adjunction property. We also have a
		natural transformation given by inclusion of $G$-invariant part
of
		sheaves on $Y$, say $I$. Now composing with the functors $\pi^*$
and
		$\pi_*$ we get another natural transformation which composed
with
		$\eta'$ gives the $\eta$ i.e. $\eta := \eta' \circ (\pi^*\cdot I
		\cdot\pi_*) $. Now to prove $\eta(\CO_X) = \Id$ we can assume
that
		$X$ is an affine variety.  Suppose $\tilde{A}$ is a structure
sheaf
		of $X$. As $A$ is flat over $A^G$ we can reduce to computing a
map
		from $\pi^* \pi^G_* (\tilde{A}) \to \tilde{A}$, in place of its
		derived functors. Now clearly the multiplication map $ A \otimes
		(_{B}A)^G \to A$ is just inverse of the natural identification
map
		of $A$ with $A \otimes (_{B}A)^G$. Hence the map $\eta(\CO_X) :
		\tilde{A} \to \tilde{A}$ is an identity map. Similarly we can
prove
		that $\mu(\CO_Y) = \Id$.
	\end{proof}

	Recall the definitions of structure sheaves and associated map of the
	sheaves given by the unital tensor functor of underlying tensor
	triangulated categories \ref{shv} i.e. given an unital functor $\pi^* :
	\CD^{per}(Y) \to \CD^G(X)$ the morphism $\Spec(\pi^*)$ induces a map of the
	structure sheaves, $\Spec(\pi^*)^{\#}:  \CO_Y \to \CO_X$. We shall prove
	that this map is an isomorphism by observing that $\Spec(\pi^*)^{\#}(V)$
	is an isomorphism for every open set $V \subseteq \Spec(\CD^{per}(Y))$. If we
	take $U = \pi^{-1}(V)$, $Z = Y \setminus V$ and $ Z' = X \setminus U$
	then we have a functor $\pi^*_V : \frac{\CD^{per}(Y)}{\CD^{per}_Z(Y)} \to
	\frac{\CD^G(X)}{\CD^G_{Z'}(X)}$ which will induce a map
	$\Spec(\pi^*)^{\#}(V):= \pi^*_V : \End_{\frac{\CD^{per}(Y)}{\CD^{per}_Z(Y)}}(\CO_Y)
	\to \End_{\frac{\CD^G(X)}{\CD^G_{Z'}(X)}}(\CO_X)$.

	\begin{lemma}
		The map  $\pi^*_V : \End_{\frac{\CD^{per}(Y)}{\CD^{per}_Z(Y)}}(\CO_Y) \to
		\End_{\frac{\CD^G(X)}{\CD^G_{Z'}(X)}}(\CO_X)$ is surjective.
	\end{lemma}
	\begin{proof}
		Suppose $[\CO_Y \xleftarrow{s} \CG \xrightarrow{a} \CO_Y]$ is an
		element of $\End_{\frac{\CD^{per}(Y)}{\CD^{per}_Z(Y)}}(\CO_Y)$ then the map
		$\pi^*$ will send it to an element $[\CO_X \xleftarrow{\pi^*(s)}
		\pi^*(\CG) \xrightarrow{\pi^*(a)} \CO_X]$ of
		$\End_{\frac{\CD^G(X)}{\CD^G_{Z'}(X)}}(\CO_X)$. It is now enough
to
		prove that this map is a bijection.

		Let  $[\CO_X \xleftarrow{t} \CF \xrightarrow{b} \CO_X] \in
		\End_{\frac{\CD^G(X)}{\CD^G_{Z'}(X)}}(\CO_X)$ be a given element
		then using the functor $\pi^G_*$ we shall get an element 
[$\CO_Y
		\xleftarrow{\pi^G_*(t)} \pi^G_*(\CF) \xrightarrow{\pi^G_*(b)}
\CO_Y
		] \in  \End_{\frac{\CD^{per}(Y)}{\CD^{per}_Z(Y)}}(\CO_Y)$ as
		$\supph(C(\pi^G_*(t))) \subseteq Z$ using the flat base change
and
		the canonical isomorphism, $i^*_V\pi^G_*(\CF) \simeq (i^*_V
		\pi_*(\CF))^G \simeq \pi^G_*(i^*_U \CF) \xrightarrow{i^*_U(t)}
		\pi^G_*(i^*_U \CO_X) \simeq \CO_V$. Now we want to prove that
		\[
			[\CO_X \xleftarrow{t} \CF \xrightarrow{b} \CO_X ] = 
[\CO_X
			\xleftarrow{\pi^*\pi^G_*(t)} \pi^*\pi^G_*(\CF)
			\xrightarrow{\pi^*\pi^G_*(b)} \CO_X]\text{.}
		\]
		Using the lemma \ref{adjunction}, we have a natural map
$\eta(\CF) :
		\pi^*\pi^G_*(\CF) \to \CF$, so to prove the assertion it is now
		enough to check that $t \circ \eta(\CF) = \pi^*\pi^G_*(t)$, $b
\circ
		\eta(\CF) = \pi^*\pi^G_*(b)$ and the cone of $\eta(\CF)$ is
		supported on $Z'$ that is $C(\eta(\CF)) \in \CD^G_{Z'}(X)$. Here
the
		first two assertions follows from the following commutative
diagrams
		which are a consequence of lemma \ref{adjunction}.
		\begin{displaymath}
			\xymatrix{
				\pi^*\pi^G_*(\CF) \ar[r]^-{\eta(\CF)}
					\ar[d]_-{\pi^*\pi^G_*(t)} & \CF
\ar[d]^-{t} &&
					\pi^*\pi^G_*(\CF) \ar[r]^-{\eta(\CF)}
					\ar[d]_-{\pi^*\pi^G_*(b)} & \CF
\ar[d]^-{b} \\
				\CO_X \ar@{=}[r]^-{\eta(\CO_X)} & \CO_X  &&
\CO_X
					\ar@{=}[r]^-{\eta(\CO_X)} & \CO_X 
			}
		\end{displaymath}
		Now the last assertion  $C(\eta(\CF)) \in \CD^G_{Z'}(X)$ is
		equivalent to $i^*_UC(\eta(\CF)) \simeq 0$ in $\CD^G(U)$
		but as the functor $i^*_U$ is exact this assertion is same as
		$C(i^*_U\eta(\CF)) \simeq 0$. Using a property of
		distinguished triangle it is enough to check that the map
		$i^*_U\eta(\CF)$ is an isomorphism. And this follows from the
		following commutative diagram.
		\begin{displaymath}
			\xymatrix@C 4 pc{
				i^*_U\pi^* \pi^G_*(\CF) \ar[d]^-{\wr}
					\ar[r]^-{i^*_U\eta(\CF)} & i^*_U\CF
\ar@{=}[d]\\
				\pi^* \pi^G_*(i^*_U \CF)
					\ar[d]^-{\wr}_-{\pi^*\pi^G_*i^*_U(t)}
					\ar[r]^-{\eta(i^*_U\CF)} &  i^*_U \CF
					\ar[d]^-{\wr}_-{i^*_U(t)} \\
				\pi^* \pi^G_*(\CO_U) \ar@{=}[r]^-{\eta(\CO_U)} &
  \CO_U
			}
		\end{displaymath}
		In above diagram we had used the same notations $\pi$ and $\eta$
		for its restriction on open subsets. Here the top left vertical
		isomorphism comes from the flat base change formula and using
		the following canonical isomorphism.
		\[
			i_U^*\pi^*\pi^{G}_*(\CF) \simeq \pi^*i^*_V
(\pi_*(\CF))^{G}
			\simeq \pi^*(i_V^* \pi_*(\CF))^{G} \simeq
			\pi^*(\pi_*i^*_U(\CF))^{G} = \pi^*\pi_*^{G}(i^*_U\CF).
		\]
		This proves that $\pi^*_V$ is surjective.
	\end{proof}
		
		%% Insert lemma
	
	\begin{lemma}
		$\pi^*_V$ is injective.
	\end{lemma}
	\begin{proof}
		Let $[\CO_Y \xleftarrow{s} \CG \xrightarrow{a} \CO_Y] \in
		\End_{\frac{\CD^{per}(Y)}{\CD^{per}_Z(Y)}}(\CO_Y)$ maps to zero in
		$\End_{\frac{\CD^G(X)}{\CD^G_{Z'}(X)}}(\CO_X)$ i.e. $[\CO_X
		\xleftarrow{\pi^*(s)} \pi^*(\CG) \xrightarrow{\pi^*(a)} \CO_X] =
0$
		which is equivalent to the existence of $\CF$ and a map $t : \CF
\to
		\pi^*\CG$ with $\supph(C(t)) \subseteq Z'$ such that $\pi^*(a)
\circ
		t = 0$. Now the map $ \pi^G_*(t) : \pi^G_*(\CF) \to
		\pi^G_*\pi^*(\CG)$ gives $\pi^G_*\pi^*(a) \circ \pi^G_*(t) = 0$
and
		as proved earlier we know that $\supph(C(\pi^G_*(t))) \subseteq
Z$
		whenever $\supph(C(t)) \subseteq Z'$.  Hence the element $[\CO_Y
		\xleftarrow{\pi^G_*\pi^*(s)} \pi^G_*\pi^*(\CG)
		\xrightarrow{\pi^G_*\pi^*(a)} \CO_Y] = 0 $ in
		$End_{\frac{\CD^{per}(Y)}{\CD^{per}_Z(Y)}}(\CO_Y)$. We shall prove that
$[\CO_Y
		\xleftarrow{s} \CG \xrightarrow{a} \CO_Y] = [\CO_Y
		\xleftarrow{\pi^G_*\pi^*(s)} \pi^G_*\pi^*(\CG)
		\xrightarrow{\pi^G_*\pi^*(a)} \CO_Y]$ as an elements of
		$\End_{\frac{\CD^{per}(Y)}{\CD^{per}_Z(Y)}}(\CO_Y)$.  Now using lemma
		\ref{adjunction} we have a map $ \mu(\CG): \CG \to
		\pi^G_*\pi^*(\CG)$ which gives the following commutative
diagrams as
		before using lemma \ref{adjunction},
		\begin{displaymath}
			\xymatrix{
				\CG \ar[r]^-{\mu (\CG)} \ar[d]_-{s} &
\pi^G_*\pi^*(\CG)
					\ar[d]^-{\pi^G_*\pi^*(s)} && \CG
\ar[r]^-{\mu (\CG)}
					\ar[d]_-{a} & \pi^G_*\pi^*(\CG)
					\ar[d]^-{\pi^G_*\pi^*(a)}\\
				\CO_Y \ar@{=}[r]^-{\mu(\CO_Y)} & \CO_Y  && \CO_Y
					\ar@{=}[r]^-{\mu(\CO_Y)} & \CO_Y 
			}
		\end{displaymath}
		Therefore it remains to prove that $i^*_VC(\mu(\CG)) = 0$ but as
		before this is equivalent to proving $C(i^*_V\mu(\CG)) =0$ since
the
		functor $i^*_V$ is an exact functor. Again using the property of
a
		distinguished triangle it is enough to prove that
$i^*_V\mu(\CG)$ is
		an isomorphism. This clearly follows from the following
commutative
		diagrams,
		\begin{displaymath}
			\xymatrix@C 4 pc{
				i^*_V\CG \ar@{=}[d] \ar[r]^-{i^*_V\mu(\CG)} &
					i^*_V\pi^G_*\pi^*(\CG) \ar[d]_-{\wr}\\
				i^*_V \CG \ar[d]^-{\wr}_-{i^*_V(s)}
\ar[r]^-{\mu(i^*_V\CG)}
					&  \pi^G_*\pi^*(i^*_V\CG)
					\ar[d]_-{\wr}^-{\pi^G_*\pi^*i^*_V(s)} \\
				\CO_V \ar@{=}[r]^-{\mu(\CO_V)} &  
\pi^G_*\pi^*(\CO_V)
				\text{.}
			}
		\end{displaymath}
		Here again as earlier the top right vertical isomorphism comes
from
		the flat base change and the following sequence of natural
		isomorphisms.
		\[
			i^*_V\pi^G_*\pi^*(\CG) \simeq i^*_V(\pi_*\pi^*\CG)^G
\simeq
			(i^*_V\pi_*\pi^*\CG)^G \simeq \pi^G_*i^*_U\pi^*\CG
\simeq
			\pi^G_*\pi^*(i^*_V\CG) \text{.}
		\]
		This proves injectivity of the map $\pi^*_V$.
	\end{proof}

	From the above two lemmas it follows that $\pi^*_V$ is an isomorphism
	and hence $\Spec(\pi^*)$ is an isomorphism of the
	varieties $\Spec(\CD^{per}(Y))$ and $\Spec(\CD^G(X))$.

\section{Example : Superschemes}\label{sec:dercatsupsch}

In this section first we shall recall the basic definition of superscheme and
some properties of it. We shall relate various notion for some superschemes with
the usual scheme with certain diagram.

\subsection{Superalgebra}

 An associative $\Z/2\Z$-grading ring is an associative ring $R$ with direct sum
decomposition $R =R^0 \oplus R^1$ as an additive group with multiplication that
preserves the grading i.e. $R^i R^j \subseteq R^{i + j}$ for $i , j \in \Z /
2\Z$. There exists a parity function which takes value in ring $\Z/2\Z =
\{0,1\}$ for every homogeneous element of $R$ i.e. if $r \in R^i$ then parity
denoted $\bar{r} = i$. Now we restrict to following important class of rings,
        \begin{definition}
         An associative $\Z/2\Z$ graded ring with unity, $R = R^0 \oplus R^1$ is
called supercommutative if the supercommutator of ring $R$ is zero i.e. $[ r_1 ,
r_2 ]:= r_1 r_2 - (-1)^{\widetilde{r_1} \widetilde{r_2}} r_2 r_1 = 0 $ for all $
r_1 , r_2 \in R$. Further ring is called $k$-superalgebra if $R$ is
supercommutative $k$-algebra with $k \subseteq R^0$.
        \end{definition}
 As usual we can define an abelian category of left modules over any
$k$-superalgebra $R$, say $\Mod(R)$. An object of this category is a $\Z / 2\Z$-
graded abelian group with left $R$-module structure which is compatible with
grading i.e. $R^i M^j \subseteq M^{i + j}$ for all $i,j = 0,1$.
Morphism between these objects is a grade preserving morphism compatible with
action of $R$. Similarly there exists a parity function defined for each
homogeneous element of module $M$ and denoted by the same notation as before. We
can define parity change functor $\Pi : \Mod(R) \to \Mod(R) ; M \mapsto \Pi M$
with $\Z / 2 \Z$ grading given by $(\Pi M)^0 = M^1 \mbox{ and } (\Pi M )^1 =
M^0$. There exists a two exact faithful functors from $\Mod(R)$ as follows,
         \[
           ff : \Mod(R) \to Vect^s_k \mbox{  and  } ff : \Mod(R) \to \Mod(R^0)
\times \Mod(R^0).
         \]
 A canonical right module structure on each left $R$ modules, say $M$, is given
by $ m r := (-1)^{\bar{m} \bar{r}} r m$. Now using this structure we can define
tensor product of two left $R$- modules, say $M_1, M_2$, as quotient of $ M_1
\otimes_{R^0} M_2$ with submodule generated by homogeneous elements $\{ r_1 m_1
\otimes m_2 - (-1)^{\bar{m_1}} m_1 \otimes r_1 m_2 / r_1 \in R^1 , m_i \in M^i
\}$. Here $M_1 \otimes_{R^0} M_2$ is defined as a tensor product of two $\Z / 2
\Z$ graded modules over a commutative ring $R^0$. A commutative constraint is
similar to the case of tensor product of supervector spaces. Note that in
general above two forgetful functors are not tensor functors. Another important
notion in commutative algebra is localisation. It is easy to define localisation
of rings and modules if multiplicative set is contained in centre of a ring.
Hence for super commutative ring we can define localisation at any homogeneous
prime ideal. It is easy to observe that given a $R$ module $M$ with a prime
ideal $\FP$, the localisation $M_{\FP} = 0$ iff $(_{R^0}M)_{\FP} = 0$ (or
$(_{(R/J)}M)_{\FP} = 0 \mbox{ where } J := R \cdotp R^1$.

\subsection{Split Superscheme}

Given any topological space $X$ we can define super ringed space as sheaf of
superring on topological space $X$. We shall denote sheaf of superring with $\Z
/ 2\Z$ grading as $\CO_X = \CO_{X,0} \oplus \CO_{X,1}$. Similarly we can define
sheaf of module and parity change functor $\Pi$ over such a ringed space as
before. We have following definition,
        \begin{definition}
        Given a ringed space $( X , \CO_X)$ is called \emph{superspace} if  ring
$\CO_X(U)$ associated to any open subset $U$ is supercommutative and each stalk
is local ring. A \emph{superspace} is called \emph{superscheme} if further
ringed space $(X , \CO_{X,0})$ is a scheme and $\CO_{X,1}$ is a coherent sheaf
over $\CO_{X,0}$. 
        \end{definition}
 We say that a superscheme is \emph{affine} if the even part of structure sheaf
$(X , \CO_{X,0})$ is affine. It is easy to see that any affine superscheme gives
a super commutative ring. Equivalently an affine superscheme associated to any
super commutative ring can be defined similar to usual affine scheme. Note in
the definition of superscheme the odd part is coherent sheaf of module over the
even part. Therefore if even part of a superscheme is noetherian then we  shall
get the left (or two sided) noetherian superscheme. Given a superscheme $( X ,
\CO_X)$ we can define sheaf of ideal\cite{Manin3} $J_X := \CO_X \cdotp
\CO_{X,1}$. Define $GrX := \oplus_{i \geq 0} J_X^i/J_X^{i + 1}$ where $J_X^0 :=
\CO_X$ and we denote the first term of $Gr X$ as $Gr_0X = \CO_X / J_X$. Now
using these notation we can define structure sheaves of \emph{even} scheme and
\emph{reduced} scheme associated to superscheme $X$ as follows,
        \[
           \CO_{X_{rd}} := Gr_0X \mbox{ and } \CO_{X_{red}} := \CO_X /
\sqrt{J_X} .
        \]
Here $J_X /J_X^2$ is a locally free sheaf of finite rank $0|d$ for some $d$ over
$\CO_{X_{rd}}$. And $Gr X$ is a Grassmann algebra over $\CO_{X_{rd}}$ of locally
free sheaf $J_X /J_X^2$.
Following particular class of superschemes are defined in paper of
Manin\cite{Manin3}.
        \begin{definition}
        A superscheme $(X , \CO_X)$ is called \emph{split} if the graded sheaf
$Gr X$ with mod 2 grading is isomorphic as a locally superringed sheaf to
structure sheaf $\CO_X$.
        \end{definition}
 Manin\cite{Manin3} had also given a way to construct such a \emph{split}
superscheme. If we take purely even scheme $(X , \CO_X)$ and a locally free
sheaf $\CV$ over $\CO_X$ then we can define 
symmetric algebra of odd locally free sheaf $\Pi \CV$, which is denoted $S(\Pi
\CV)$ (see Manin\cite{Manin3}), then $(X , S(\Pi \CV))$ is a split superscheme.
An important example is given by projective superscheme $ \mathbb P^{m|n}$ which
is given by locally free sheaf $\CO(-1)^n$. An example of a nonsplit superscheme
given in Manin\cite{Manin3} is Grassmann superscheme $G(1|1 ,\C^{2|2})$ which is
also an example of non superprojective scheme.
%
%Add some more properties and examples later
%
We can define an abelian category of sheaf of left modules over $\CO_X$, denoted
$\Mod^s(X) \mbox{ or } \Mod(\CO_X)$. As above we have a natural right module
structure given by the Koszul sign rule.
When $(X , \CO_X)$ is affine superscheme given by super ring  $R$ then we can
define the sheaf of module associated to any $R$-module $M$ similar to
commutative case. Hence we can define quasi-coherent and coherent sheaves over
any superscheme. Therefore we shall get two abelian subcategories namely
category of all quasi-coherent sheaves and coherent sheaves. We denote them by
$\qcoh(\CO_X)$ and $\coh(\CO_X)$ respectively. Now similar to above we have
forgetful functor as follows,
        \[
          ff : \Mod(\CO_X) \to \Mod(\CO_{X,0}) \times \Mod(\CO_{X,0}).
        \]
It is an exact faithful functor. we can easily see that 
        \beast
          \qcoh(\CO_X) &=& ff^{-1}(\qcoh(\CO_{X,0}) \times \qcoh(\CO_{X,0}))  \\
          \coh(\CO_X) &=& ff^{-1}(\coh(\CO_{X,0}) \times \coh(\CO_{X,0})).
        \eeast
We can define the tensor product of two sheaves of modules over superscheme
similar to usual scheme. We shall use the canonical identification of sheaf of
left and right modules by Koszul sign rule. Define tensor product of two sheaves
of modules $\CF_1$ and $\CF_2$ as the sheaf associated to pre sheaf given by 
         \[
           U \mapsto (\CF_1 \otimes \CF_2)(U) := \CF_1(U) \otimes_{\CO_X(U)}
\CF_2(U).
         \]
Note that with this definition of tensor structure the commutative constraint is
given by sign rule i.e. $(-1) : \Pi \CF \otimes \Pi \CG \to \Pi \CG \otimes \Pi
\CF$ for purely odd sheaves and identity for other sheaves. Now we can prove
some easy properties of this tensor product by just reducing to affine case,
         \begin{proposition}
         \label{denseness}
         Suppose $(X , \CO_X)$ is a split superscheme and $\CF$ and $\CG$ are
quasi coherent sheaves. 
            \begin{enumerate}
             \item Any $\CO_{X_{rd}}$ quasi coherent sheaf $\CF^0$ is also a
$\CO_X$ quasi coherent sheaf via canonical projection $\CO_X \to \CO_{X_{rd}}$.
Hence we get a functor ${\bf i_{rd}} : \CD_{qc}(X_{rd}) \to \CD_{qc}(X)$. 
             \item The functor ${\bf i_{rd}} $ is a tensor functors and the
images of this functor is tensor ideals in $\CD_{qc}(X)$. The functor ${\bf
i_{rd}}$ is in fact a dominant tensor functor.
             \item $(\Pi \CF) \otimes \CG = \CF \otimes (\Pi \CG) = \Pi(\CF
\otimes \CG)$.
             \end{enumerate}
         \end{proposition}
         \begin{proof}
         The proofs of \emph{1} and \emph{3} are clear from the definition.
Hence we just indicate the proof of \emph{2}.\\
         \emph{ Proof of 2.} Given any quasi coherent sheaf $\CF$, observe that
${\bf i_{rd}}(\CF)$ has the trivial action of ideal sheaf $J_X$. Therefore by
definition of tensor product it follows that ${\bf i_{rd}}$ is a tensor functor.
Also observe that given a sheaf of $\CO_X$ module, $\CF$, the tensor $\CF
\otimes_{\CO_X} \CO_{X_{rd}}$ has the trivial action of the ideal sheaf $J_X$
and hence it will be in the image of the functor ${\bf i_{rd}}$. Since $(X
,\CO_X)$ is a split superscheme, we have identification of $\CO_X$ with $Gr X$.
The sheaf $Gr X$ is an exterior algebra over purely odd locally free sheaf $\Pi
\CV := J_X / J_X^2$ and each subquotients$J_X^i / J_X^{i + 1}$  can be
identifies with $\Pi^i \Lambda^i \CV$. Hence each subquotients are purely odd or
purely even locally free sheaves. The $\Z$-grading on sheaf $Gr X$ gives a
filtration for structure sheaf $\CO_X$ and hence we have following Postnikov
tower for each complex of quasi coherent sheaf $\CF$,
             \begin{displaymath}
                \xymatrix@C=4pt@M=.1pt{
                          \CF \ar[rd]& & \ar[ll] J_X \otimes \CF &\cdots &
J_X^{n-1} \otimes \CF \ar[rd]& &\ar[ll] J_X^n \otimes \CF \ar@{=}[rd]& \\
                          & \CO_{X_{rd}} \otimes \CF \ar@{-->}[ru] & &\cdots &&
\Pi^{n-1} \Lambda^{n-1} \CV \otimes \CF \ar@{-->}[ru] &  & \Pi^n \Lambda^n \CV
\otimes \CF.
                         }
             \end{displaymath}
         In above tower the lower order term is complex of either purely odd or
purely even sheaves. And using \emph{3}, we have $ \Pi^i \Lambda^i \CV \otimes
\CF = (\Pi^i \CO_{X_{rd}}) \otimes ( \Lambda^i \CV \otimes \CF)$. Therefore the
ideal generated by the image of the functor ${\bf i_{rd}}$ contains the all
lower order term of above Postnikov tower and hence ${\bf i_{rd}}$ is a dominant
tensor functor.
         \end{proof}
 Given a split superscheme $(X , \CO_X = S^.(\Pi \CV) = \Pi \Lambda^.(\CV))$
there is one more forgetful functor as follows,
         \[
           ff : \Mod(\CO_X) \to \Mod(\CO_{X_{rd}}) \times \Mod(\CO_{X_{rd}}) .
         \]
This functor is defined using the obvious inclusion of $\CO_{X_{rd}}$ inside
$\CO_X$ which comes from the definition of split superscheme. For simplicity
take $\Lambda^{\sharp} := \Lambda^{\sharp}(\CV)$ for $\sharp = ev \mbox{ or }
odd $. Suppose $\eta : \Lambda^{\sharp} \otimes \Lambda^{\flat} \to
\Lambda^{\sharp + \flat}$, where tensoring is over $\CO_{X_{rd}}$ which we omit
writing later also and $\sharp + \flat$ is defined evidently, and where $\sharp
= ev \mbox{ or } odd \mbox{ and } \flat = ev \mbox{ or } odd$, represents the
natural multiplication of subsheaves of Grassmann algebra. Given a $\CO_X$
module $\CF = \CF^0 \oplus \CF^1$ we have multiplication structure of $\CO_X$
which can be described using following maps,
         \begin{eqnarray}
         \label{multiplication}
           m^{ev}_i : \Lambda^{ev}  \otimes \CF^i \to \CF^i  \mbox{ and }  
           m^{odd}_i : \Lambda^{odd} \otimes \CF^i \to \CF^{i +1} \mbox{ where }
i \in \Z / 2\Z,
         \end{eqnarray}
 and certain commutative diagrammes, where $i \in \Z / 2\Z$,
         \begin{displaymath}
            \xymatrix{
                     \Lambda^{ev}  \otimes \Lambda^{ev} \otimes \CF^i \ar[r]^-{1
\otimes m^{ev}_i} \ar[d]_{\eta \otimes 1} & \Lambda^{ev} \otimes \CF^i
\ar[d]^{m^{ev}_i}  &  \Lambda^{odd}  \otimes \Lambda^{odd} \otimes \CF^i
\ar[r]^-{1 \otimes m^{odd}_i} \ar[d]_{\eta \otimes 1} & \Lambda^{odd} \otimes
\CF^{i + 1} \ar[d]^{m^{odd}_{i+1}} \\
                     \Lambda^{ev} \otimes \CF^i \ar[r]^{m^{ev}_i}  & \CF^i &
\Lambda^{ev} \otimes \CF^i \ar[r]^{m^{ev}_i}  & \CF^i 
                     }
         \end{displaymath}
         \begin{displaymath}
           \xymatrix{  
                      &  \Lambda^{ev}  \otimes \Lambda^{odd} \otimes \CF^i
\ar[r]^-{1 \otimes m^{odd}_i} \ar[d]_{\eta \otimes 1} & \Lambda^{ev} \otimes
\CF^{i + 1} \ar[d]^{m^{ev}_{i + 1}}     &  \\
                        &  \Lambda^{odd} \otimes \CF^i \ar[r]^{m^{odd}_i} &
\CF^{i +1} & .
                    }
         \end{displaymath}
 It is now easy to see that pair of sheaves $(\CF^0 , \CF^1)$ of $\CO_{X_{rd}}$
modules with maps and commutative diagrammes as above\ref{multiplication} will
give the sheaf of $\CO_X$ module structure on $\CF := \CF^0 \oplus \CF^1$. Using
property of forgetful functor we can prove that any quasi-coherent (resp.
coherent) sheaf of $\CO_X$ module comes from a pair of quasi coherent (resp.
coherent) sheaves of $\CO_{X_{rd}}$ modules with the data of maps and diagrammes
as above\ref{multiplication}. Note also that the Grassmann algebra constructed
from locally free sheaf $\CV$ gives a locally free sheaf of $\CO_{X_{rd}}$
module. Therefore structure sheaf $\CO_X$ is locally free sheaf as a
$\CO_{X_{rd}}$ module.
 
 Similar to usual scheme we can take $\CD(X):= \CD(\Mod(X))$ the derived
category of abelian category $\Mod(X)$. There are various triangulated
subcategories like $\CD_{qc}^{\sharp}(X):= \CD^{\sharp}(\qcoh(X))$ and
$\CD_{coh}^{\sharp}(X) := \CD^{\sharp}(\coh(X))$ where $\sharp = + , - , b
\mbox{ or } \emptyset$. For convenience we shall denote by $\CD^{\sharp}(X^0) :=
\CD^{\sharp}(\Mod(\CO_X^0))$ (resp. $\CD^{\sharp}(X_{rd}) :=
\CD^{\sharp}(\Mod(\CO_{X_{rd}}))$ ) the derived category of modules over purely
even scheme $(X , \CO_X^0)$ (resp. $X_{rd} = (X , Gr_0X)$). Similar notation we
can have for the other subcategories. We shall now define the another important
triangulated subcategory of $\CD_{qc}(X)$.
         \begin{definition}
         Given a complex $\CF^.$ of quasi coherent sheaves of modules over
superscheme $(X , \CO_X)$ is called \emph{strictly perfect} if $\CF^.$ is quasi
isomorphic to bounded complex of locally free coherent sheaf of $\CO_X$ module.
A complex $\CF^.$ is called \emph{perfect} if it is locally quasi isomorphic to
bounded complex of locally free coherent sheaves. 
         \end{definition}
 We shall denote the triangulated subcategory of all perfect complexes as
$\CD^{per}(X) \subseteq \CD_{qc}(X)$. Similar to scheme case we can extend
various functors at the level of these triangulated categories. Hence we can
prove $\CD^{per}(X)$ is a tensor triangulated category with tensor given by
derived functor of usual tensor product defined as above. We can extend the
forgetful functor defined earlier using exactness,
         \[
           ff : \CD^{\sharp}(X) \to \CD^{\sharp}(X^0) \times \CD^{\sharp}(X^0)
\mbox{ and }
           ff : \CD^{\sharp}_{qc}(X) \to \CD^{\sharp}_{qc}(X^0) \times
\CD^{\sharp}_{qc}(X^0) .
         \]
 Here $\sharp \in \{ + , - , b , \emptyset \}$. We can have similar forgetful
functors for the case of coherent sheaves. If we restrict to split superscheme
then we can also have forgetful functor for the case of locally free sheaves (or
vector bundles). Hence for a split superschemes we have following  forgetful
functor for the triangulated subcategory of perfect complexes,
         \[
           ff : \CD^{per}(X) \to \CD^{per}(X_{rd}) \times \CD^{per}(X_{rd})
         \]
Note that this functor may not be a tensor functor.
\subsection{Main Results}
%  Check it later again!!!!!!!!!

Now we have following result which gives a way to get back quasi coherent
complexes over superscheme with a pair of quasi coherent complexes over
\emph{purely even superscheme} (or \emph{usual scheme}).
         \begin{lemma}
         \label{sheaf diag}
           Given a split superscheme $(X , \CO_X)$, take two quasi coherent
complexes $\CF^0$ and $\CF^1$ over purely even superscheme $X_{rd}$. Suppose we
have maps $m^{ev}_i : \Lambda^{ev}  \otimes \CF^i \to \CF^i  \mbox{ and }
m^{odd}_i : \Lambda^{odd} \otimes \CF^i \to \CF^{i +1}, \mbox{ where } i \in \Z
/ 2\Z$,  with commutative diagrammes as before\ref{multiplication} then $\CF :=
\CF^0 \oplus \CF^1$ has a structure of quasi coherent complex over $X$.
         \end{lemma}
         \begin{proof}
            Since $\CV$ is locally free sheaf therefore giving a multiplication
structure at the level of complexes is same as giving a complex of  quasi
coherent sheaves with such a multiplication structure.
         \end{proof}
 Now similar to usual scheme we can define \emph{support} of a quasi coherent
sheaf as a closed subset of $X$ containing all super prime ideals where stalk of
the sheaf is nonzero. Since nontriviality of a stalk at any point $\FP$ is a
local property we can check it in an affine open set containing $\FP$. Now from
earlier observation $\CF_{\FP} = 0$ iff $\CF_{\FP}^0 = 0 = \CF_{\FP}^1$ as a
stalk of a sheaf of $\CO_{X_{rd}}$ modules $\CF^0$ and $\CF^1$. Therefore for a
quasi coherent sheaf $\CF$ we have $ supp (\CF) = supp ( ff (\CF)) = supp(\CF^0)
\cup supp(\CF^1)$. Now the assignment of support can be extended to derived
category as follows,
         \[
           supph(\CF^.) := \cup_{i \in \Z}  \CH^i(\CF^.) .
         \]
This association can be restricted to thick subcategory $\CD^{per}(X)$. As
forgetful functor is an exact functor  we have following relation of support
similar to sheaf case,
         \[
           supph(\CF^.) = supph (ff(\CF^.)) = supph(\CF^0) \cup supph(\CF^1)
         \]
Using this property of support we can prove following result,
         \begin{lemma}
            The pair $(X , supph)$ defined as above gives a support data on a
triangulated category $\CD^{per}(X)$. 
         \end{lemma}
         \begin{proof}
            Since forgetful functor is an exact functor and we have equality
$(\CF^.) = supph(ff(\CF^.))$ therefore the support data properties (SD 1)-(SD
4)\cite{PB2} are easy to prove. We shall just prove (SD 5) here. Again checking
nontriviality of stalk is a local question, we can assume that $X$ is an affine
superscheme. First we observe that any perfect complex $\CF^.$ is just a strict
perfect and hence bounded complex of projective modules. Therefore using
induction on lengths of complexes we can reduce to proving the statement for a
modules $M$ and $N$. Since trivially $M_{\FP} = 0$ or $N_{\FP} = 0$ gives
$M_{\FP} \otimes N_{\FP} = 0$, it is enough to prove $ M_{\FP} \neq 0 \mbox{ and
} N_{\FP} \neq 0 \Rightarrow M_{\FP} \otimes N_{\FP} \neq 0$. But by taking two
elements $m \in M$ and $n \in N$ with $ann(m) \cap (R - \FP) = ann(n) \cap (R -
\FP) \emptyset$ we can easily see that $ ann(m \otimes n) \cap (R - \FP)=
\emptyset$. Hence we have $ M_{\FP} \otimes N_{\FP} \neq 0$ and this will prove
(SD 5).
          \end{proof}
 We shall prove now that above support data is in fact a classifying support
data as defined in Balmer\cite{PB2}. We need following classification of thick
tensor subcategories\cite{PB2} of $\CD^{per}(X)$ which we prove by relating it
with the usual scheme case.
         \begin{proposition}
            Given a split superscheme $(X , \CO_X)$ we have a following
bijection,
            \[
              \theta : \{ Y \subset X | Y \mbox{ specialisation closed}\}
\xrightarrow{\sim} \{ \CI \subset \CD^{per}(X) | \CI \mbox{ radical thick
$\otimes$-ideal} \}
            \]
            defined by $ Y \mapsto \{ \CF^. \in \CD^{per}(X) | supph(\CF^.)
\subset Y \}$, with inverse, say $\eta$, $\CI \mapsto supph(\CI):= \cup_{\CF^.
\in \CI} supph(\CF^.)$.
         \end{proposition}
         \begin{proof}
            Using support data properties (SD 1) - (SD 5) we can prove that
$\theta(Y)$ is a radical thick tensor ideal and hence the map $\theta$ is well
defined. To prove that $\eta(\CI)$ is a specialisation closed subset it is
enough to prove that for any $y \in \eta(\CI)$ there is a closed set containing
this point. By definition $y$ is in homological support of some object $\CF^.
\in \CI$. Hence $y \in supph(ff(\CF^.))$ which is a closed subset.\\
            It is easy to check that $ \eta \circ \theta ( Y ) \subseteq Y $ and
$ \CI \subseteq \theta \circ \eta ( \CI ) $. To prove that $ Y \subseteq \eta
\circ \theta ( Y ) $ it is enough to say that for any closed subset $ Z $ there
exists an object with support $ Z $. But there exists a $\CO_{X_{rd}}$  perfect
sheaf with support $Z$ and hence via natural map $\CO_X \to \CO_{X_{rd}}$ we get
a perfect sheaf with support $Z$.\\
            Finally to prove that $ \theta \circ \eta ( \CI ) \subseteq \CI $ it
is enough to prove that for any $\CF^. \in \theta \circ \eta (\CI)$ the object
$\CF^. \in \CI$. Now following proof of theorem 3.15 of Thomason\cite{Thomason2}
it reduces to proving that $\supph (\CF^.) \subseteq \supph(\CG^.)$ for some
object of $\CG^. \in \CI$ then $\CF^. \in \CI$. But $\CF^. \otimes \CO_{X_{rd}}$
will be in thick tensor ideal generated by $\CG^. \otimes \CO_{X_{rd}}$ as there
is a dominant tensor inclusion of $\CD^{per}(X_{rd})$ in
$\CD^{per}(X)$\ref{denseness}. This will prove $\CF^. \otimes \CO_{X_{rd}} \in
\CI$. But $\CI$ is intersection of prime ideal containing $\CI$ and
$\CO_{X_{rd}}$ is not in any prime ideal and hence $\CF^. \in \CI$   
         \end{proof}
With this result it follows that  $(X , supph)$ is a classifying support data on
a tensor triangulated category $\CD^{per}(X)$ as other properties mentioned in
Balmer\cite{PB2} is clearly holds for $X$. Using Theorem 5.2 of Balmer\cite{PB2}
we get following corollary,
         \begin{corollary}
         \label{top spc}
          The canonical map $f : X \to \Spc(\CD^{per}(X))$ given by $x \mapsto
\{ \CF^. \in \CD^{per}(X) | x \notin supph(\CF^.) \}$ is a homeomorphism.
         \end{corollary}

Now we shall prove the localisation theorem similar to Thomason for split
superscheme case by using the generalisation of Thomason result proved by
Neeman\cite{Neeman3}. First we recall some notations. Given a closed subset $Z$
of $X$ we can define full triangulated subcategory $\CD_{qc , Z}(X) \subseteq
\CD_{qc}(X)$ containing all objects with homological support contained in closed
subset $Z$. Suppose $U$ is an open complement of closed subset $Z$. There is a
canonical restriction functor $j^* : \CD_{qc}(X) \to \CD_{qc}(U)$ and clearly it
will be trivial functor on thick subategory $\CD_{qc , Z}(X)$. Now using
forgetful functor we have following commutative diagram,
         \[
           \xymatrix@C 1 pc {
                     \CD_{qc , Z}(X) \ar[r] \ar[d]^{ff} & \CD_{qc}(X)
\ar[r]^{j^*} \ar[d]^{ff} & \CD_{qc}(U) \ar[d]^{ff} \\
                     \CD_{qc , Z}(X_{rd}) \times \CD_{qc , Z}(X_{rd}) \ar[r] & 
\CD_{qc}(X_{rd}) \times \CD_{qc}(X_{rd}) \ar[r] &  \CD_{qc}(U_{rd}) \times
\CD_{qc}(U_{rd}). 
                    }
         \]

 We have following result which we shall prove using usual scheme case as
before,
         \begin{proposition}
         \label{local}
          The canonical functor induced from the functor $j^*$, say $j^*
:\CD_{qc}(X) / \CD_{qc , Z}(X) \xrightarrow{\sim} \CD_{qc}(U) $ is an
equivalence.
         \end{proposition}
         \begin{proof}
         We shall prove that $j^*$ is fully faithful and essentially surjective.
Recall that the quotient functor $j^*$ gives a map between morphisms as
follows, 
         \[
           [\CF \leftarrow^s \CF' \rightarrow^a \CG ] \mapsto [j^*\CF
\leftarrow^{j^* s} j^*\CF' \rightarrow^{j^*a} j^*\CG].
         \]
         Now using the forgetful functor we can get a similar map between
morphisms of $\CO_{X_{rd}}$ perfect complexes.\\
         To prove faithfulness suppose there exists a morphism $\tilde{t} :
\tilde{\CF''} \to \CF' $ with $ a \circ \tilde{t} = 0 $ and cone of $\tilde{t}$
is an object of $\CD_{qc , Z}(X)$. Since we have equivalence of functor $j^*$
after applying forgetful functor therefore there exists an object $\CF' :=
\CF''^0 \oplus \CF''^1$ and a map $t:= t^0 \oplus t^1 : \CF'' \to \CF'$ with
$ff(a) \circ t = 0$. Now  using lemma\ref{sheaf diag} it is enough to prove that
$\CF''$ has multiplication structure and the map $t$ is compatible with it.
Again using the fullness of $ff(j^*)$, the multiplicative structure on
$\tilde{\CF''}$ can be lifted to $\CF''$. Now using faithfulness of $ff(j^*)$ we
get the commutativity of various digram for multiplicative structure. Since
$\tilde{t}$ is compatible with multiplicative structure on $\tilde{\CF''}$,
again faithfullness gives compatibility of the multiplicative structure  on
$\CF''$ with the map $t$. Since support doesn't change under forgetful functor
we get the cone of $t$ as an object of $\CD_{qc , Z}(X)$ and $ a \circ t = 0$.\\
         We shall use similar idea as above to prove fullness of the functor
$j^*$. So given any map $ [j^*\CF \leftarrow^{\tilde{s}} \tilde{\CF'}
\rightarrow^{\tilde{a}} j^*\CG] $ we want it to be image of some map $[\CF
\leftarrow^s \CF' \rightarrow^a \CG ]$  under $j^*$. But there exists a map $
[ff(\CF) \leftarrow^s \CF'^0\oplus \CF'^1 \rightarrow^a ff(\CG) ]$ which maps to
above map via $ff(j^*)$. As above using lemma\ref{sheaf diag} it is enough to
give multiplicative structure on $\CF' := \CF'^0 \oplus \CF'^1$ which is
compatible with maps $s$ and $a$. Now using fullness of the functor $ff(j^*)$ we
can prove existence of multiplication map and using faithfulness we can see
commutativity of various diagram to lift multiplicative structure from
$\tilde{\CF'}$. Once again using faithfulness of $ff(j^*)$ we can lift the
compatibility of multiplicative structure with the maps $s$ and $a$.\\
         Now to prove essential surjectivity of functor $j^*$, we start with an
object $\tilde{\CF} \in \CD_{qc}(U)$. Since $ff(j^*)$ is essentially surjective
we get an isomorphism of $ff(\tilde{\CF})$ with an object $\tilde{\CF'^0} \oplus
\tilde{\CF'^1}$ which is an image of the object $\CF'^0 \oplus \CF'^1$ via the
functor$ff(j^*)$. We can give the multiplicative structure on $\tilde{\CF'^0}
\oplus \tilde{\CF'^1}$ s.t it becomes isomorphic to $\tilde{\CF} \in
\CD_{qc}(U)$. Now we can lift this multiplicative structure to $\CF' := \CF'^0
\oplus \CF'^1$ using fully faithfulness of $ff(j^*)$. Hence $j^*$ is essentially
surjective.
         \end{proof}
Another notion which we need is compact object in a triangulated category and
compactly generated triangulated category.
         \begin{definition}
         \begin{enumerate}
            \item[(a)] An object $t$ in a triangulated category, which is closed
under formation of every small coproducts, is said to be \emph{compact} if
$Hom(t, \_ )$ respects coproducts. In a triangulated category $\CT$, the full
subcategory of all compact objects is denoted as $\CT^c$.
            \item[(b)] A triangulated category $\CT$, which is closed under
formation of every small coproducts, is said to be \emph{compactly generated} if
there exists a small set $T$ of compact objects s.t. $\CT$ is a smallest
triangulated subcategory closed under coproducts and distinguished triangles
containing $T$. Equivalently, $\CT$ is called \emph{compactly generated} iff
$T^{\perp} := \{ x \in \CT | Hom_{\CT}(t ,x) = 0 \mbox{ for all } t \in T  \} =
0$. The set of compact objects $T$ is called \emph{generating set} if further
$T$ is closed under suspension or translation. 
         \end{enumerate}
         \end{definition}
          Now we shall recall the theorem 2.1 of Neeman\cite{Neeman1} which is
proved in quite generality and is a slight strengthening of theorem 2.1 of
Neeman\cite{Neeman3}.
         \begin{theorem}[Neeman\cite{Neeman3}\cite{Neeman1}] 
          \label{Neeman}
          Let $\CS$ be a compactly generated triangulated category. 
          Let $R$ be set of compact objects of $\CS$ closed under suspension and
$\CR$ be a localising subcategory generated by $R$ in triangulated category
$\CS$. Under these hypothesis\cite{Neeman3} there exists the quotient category
$\CT$ of $\CS$ with adjoint functor of natural functor $j^*$ i.e. there is
following sequence of triangulated categories,
              \[
                \CR \rightarrow \CS \xrightarrow{j^*} \CT . 
              \]
         It induces a functors at the level of subcategories of compact objects
i.e.
              \[
                \CR^c \rightarrow \CS^c \xrightarrow{j^*} \CT^c . 
              \]
         \begin{enumerate}
         \item The category $\CR$ is compactly generated, with $R$ as a
generating set.
         \item If $R$ is generating set for all of $\CS$ then $\CR = \CS$.
         \item If $R \subset \CR$ is closed under the formation of triangles and
direct summands, then it is all of $\CR^c$. In any case $\CR^c = \CR \cap
\CS^c$. 
         \item The induced functor $ F :\CS^c / \CR^c \to \CT^c$ is fully
faithful and every object of $\CT^c$ is isomorphic to direct summand of image of
the functor $F$. In particular, if $\CT^c$ is an idempotent complete then we get
an equivalence from idempotent completion $\widetilde{\CS^c / \CR^c}$ to
triangulated category $\CT^c$.
         \end{enumerate}
         \end{theorem}
In our particular situation we take $ \CS := \CD_{qc}(X) , \CR := \CD_{qc ,
Z}(X)$ and as we proved above\ref{local} the quotient will be $\CT :=
\CD_{qc}(U)$. We shall now prove following result which will provide all
hypothesis for application of above theorem.
         \begin{proposition}
         following statements are true for any split superscheme $(X , \CO_X)$
            \begin{enumerate}
             \item The triangulated category $\CD_{qc}(X)$ is closed under
formation of every small        coproducts.
             \item The triangulated category $\CD_{qc}(X)$ is a compactly
generated category.
             \item $\CD_{qc , Z}(X)^c \simeq \CD^{per}_Z (X)$ for any closed
subset $Z$ of $X$.
            \end{enumerate}
         \end{proposition}
         \begin{proof}
         \emph{ Proof of 1.} This is similar to usual scheme case, see example
1.3 of Neeman\cite{Neeman1}.

         \emph{Proof of 2.} Take a small set of objects $T \subset \CD_{qc}(X)$
of all perfect complexes of $\CO_{X_{rd}}$ modules via functor ${\bf i_{rd}}$
and its full image under the functor $\Pi$. Now using the canonical filtration
and result\ref{denseness} we have following Postnikov tower for every object
$\CF \in \CD_{qc}(X)$,
            \begin{displaymath}
             \xymatrix{
                      \CF \ar[rd]& & \ar[ll] \CG^1 &\cdots & \CG_{n-1}\ar[rd]&
&\ar[ll] \CG_n \ar@{=}[rd]& \\
                       & \CF_1 \ar@{-->}[ru] & &\cdots && \CF_{n-1}
\ar@{-->}[ru] &  & \CF_n
                      }.
            \end{displaymath}
         Now the base of above tower,$\CF_i := \CF \otimes_{\CO_X} \Pi^i
\Lambda^i(\CV) \in Im({\bf i_{rd}})$, is generated by objects of the set $T$ and
hence every object $\CF \in \CD_{qc}(X)$ is generated by the set $T$.
         
         \emph{ Proof of 3.} It is enough to prove that all perfect complexes
are compact objects. Indeed, the full subcategory of perfect complexes is closed
under triangles and direct summands similar to usual scheme. Hence by taking $R$
to be all perfect complexes the above result of Neeman\ref{Neeman} proves that
all compact objects are perfect complexes. Now to prove that every perfect
complex is compact object we have to first observe following,
             \[
                (H^0(\CR \CH om(\CF , \CG)))^0 = Hom_{\CO_X}(\CF , \CG).
             \]
         Here $\CR \CH om(\CF , \CG)$ is an internal homomorphisms between $\CF$
and $\CG$. Now rest of the proof is similar to the proof given in example 1.13 of
Neeman\cite{Neeman1}.
         \end{proof}

Using the above result\ref{Neeman} it is easy to deduce following corollary,
         \begin{corollary}
          \label{reconstruction}
          Given a split superscheme $(X , \CO_X)$ we have an equivalence of
tensor triangulated categories, $F : \widetilde{\CD^{per}(X) /  \CD^{per}_Z(X)}
\xrightarrow{\sim} \CD^{per}(U)$.
         \end{corollary}
         \begin{proof}
          It is enough to observe that $j^*$ induces a tensor functor.
         \end{proof}
Similar to Balmer\cite{PB2} we shall use above localisation result to give
relation between structure sheaves. Balmer\cite{PB2} had defined structure sheaf
of $\Spc(\CK)$ for any tensor triangulated category as a sheaf associated to the
presheaf given by $ U \mapsto End_{\CK / \CK_Z}(1_U)$ where $U$ is an open set
and $1_U \in (\CK / \CK_Z)$ is the image of tensor unit $1 \in \CK$. Define
$\Spec(\CD^{per}(X)) := ( \Spc(\CD^{per}(X)) , \CO_{\CD^{per}(X)})$ the locally
ringed space associated to tensor triangulated category $\CD^{per}(X)$. Now the
homeomorphism $f$\ref{top spc} defined above for split superscheme gives a map
of locally ringed space, $f : (X \simeq X^0 , \CO_{X^0}) \to \Spec(\CD^{per}(X))
$. Here the map of structure sheaves comes from the identification given in
corollary\ref{reconstruction}. We have following result similar to Theorem 6.3
of Balmer\cite{PB2},
         \begin{theorem}
         Suppose $X$ is a topologically noetherian (that is, if all open subset
are quasi compact) split superscheme. The map $f$ defined as above gives an
isomorphism of locally ringed space i.e. $X^0 \simeq \Spec(\CD^{per}(X))$.
         \end{theorem}
         \begin{proof}
          Using the homeomorphism $f$ it is enough to prove isomorphism of
structure sheaves. Hence we can assume that superscheme is affine. Now using the
remark 8.2 of Balmer\cite{PB1} and localisation theorem\ref{reconstruction} we
can prove that induced map of sheaves is an isomorphism.
         \end{proof}

\bibliography{ref}
\bibliographystyle{plain}
\end{document}